\documentclass[11pt,oneside,reqno]{amsart}

\usepackage{hyperref}

\usepackage[headsep=30pt,footskip=30pt,twoside=false,bottom=2.5cm]{geometry}
\usepackage{tcolorbox}

\allowdisplaybreaks[1]

\usepackage{etoolbox}

\usepackage[all]{xy} \usepackage{tikz-cd} \xymatrixcolsep{2cm}

\makeatletter

\newtheorem{theorem}{Theorem}[section]

\newtheorem{proposition}[theorem]{Proposition}

\newtheorem{lemma}[theorem]{Lemma}

\newtheorem{corollary}[theorem]{Corollary}

\theoremstyle{remark} 

\newtheorem{remark}[theorem]{Remark}

\theoremstyle{definition} \newtheorem{definition}[theorem]{Definition}

\newtheorem{example}[theorem]{Example}

\newcommand{\Hom}[3][]{\ensuremath{\mathrm{Hom}_{#1} (#2, #3)}}
\newcommand{\Aut}[2][]{\ensuremath{\mathrm{Aut}_{#1} (#2)}}

\newcommand{\id}[1]{\ensuremath{\mathbf{1}_{#1}}}

\newcommand{\set}[1]{\ensuremath{\{ #1 \}}}
\newcommand{\suchthat}{\ensuremath{\, \vert \,}}

\newcommand{\N}{\ensuremath{\mathbf{N}}}

\newcommand{\Z}{\ensuremath{\mathbf{Z}}}

\newcommand{\R}{\ensuremath{\mathbf{R}}}

\newcommand{\C}{\ensuremath{\mathbf{C}}}

\renewcommand{\bar}[1]{\ensuremath{\overline{#1}}}

\newcommand{\units}[1]{\ensuremath{#1^\times}}
\renewcommand{\dim}[2][]{\ensuremath{\mathrm{dim}_{#1}(#2)}}

\newcommand{\Repr}[2][]{\ensuremath{\mathbf{Rep}_{#1}(#2)}}

\newcommand{\pr}[1]{\ensuremath{\mathrm{pr}_{#1}}}

\newcommand{\card}[1]{\ensuremath{\mathrm{card}(#1)}}

\newcommand{\struct}[1]{\ensuremath{\mathcal{O}_{#1}}}

\newcommand{\rk}[2][]{\ensuremath{\mathrm{rk}_{#1}(#2)}}

\newcommand{\restrict}[2]{\ensuremath{#1\vert_{#2}}}
\newcommand{\sr}{\ensuremath{\mathrm{s}}}
\newcommand{\tg}{\ensuremath{\mathrm{t}}}
\newcommand{\stab}[1]{\ensuremath{#1^{\mathrm{s}}}}
\newcommand{\sstab}[1]{\ensuremath{#1^{\mathrm{ss}}}}
\definecolor{dgreen}{rgb}{0,0.5,0}
\definecolor{dbrown}{rgb}{0.4, 0.26, 0.13}

\begin{document}

\title[Tensor product of quiver representations]{Tensor product of representations of quivers}

\author{Pradeep Das}
\email{pradeepdas0411@gmail.com}
\address{Rajiv Gandhi Institute of Petroleum Technology, Jais\\ Harbanshganj \\ Amethi 229304 \\ India}
\author{Umesh V.~Dubey and N.~Raghavendra}
%
%
%
\email{umeshdubey@hri.res.in, raghu@hri.res.in}
\address{Harish-Chandra Research Institute \\  A CI of Homi Bhabha National
Institute \\ Chhatnag Road \\ Jhunsi \\ Prayagraj - 211019 \\ India}

\begin{abstract}
In this article, we define the tensor product $V\otimes W$ of a
representation $V$ of a quiver $Q$ with a representation $W$ of an
another quiver $Q'$, and show that the representation $V\otimes W$ is
semistable if $V$ and $W$ are semistable.  We give a relation between
the universal representations on the fine moduli spaces $N_1, N_2$ and
$N_3$ of representations of $Q, Q'$ and $Q\otimes Q'$ respectively
over arbitrary algebraically closed fields. We further describe a
relation between the natural line bundles on these moduli spaces when
the base is the field of complex numbers.  We then prove that the
internal product $\tilde{Q}\otimes \tilde{Q'}$ of covering quivers is
a sub-quiver of the covering quiver $\widetilde{Q\otimes Q'}$.  We
deduce the relation between stability of the representations
$\widetilde{V\otimes W}$ and $\tilde{V} \otimes \tilde{W}$, where
$\tilde{V}$ denotes the lift of the representation $V$ of $Q$ to the
covering quiver $\tilde{Q}$.  We also lift the relation between the
natural line bundles on the product of moduli spaces
$\tilde{N_1} \times \tilde{N_2}$. 
 
\end{abstract}

\maketitle


\section*{Introduction}
Whenever the tensor product of two objects exists in a category on
which there is a notion of semistability, it is natural to ask whether
the tensor product of two semistable objects is again semistable.  In
the case of holomorphic vector bundles over a compact Riemann surface,
it follows from a theorem of Narasimhan and Seshadri \cite{NS} that the
tensor product of semistable vector bundles is semistable.  To every
filtered complex vector space, Faltings and W\"ustholz \cite{FW}
associated a holomorphic vector bundle over a compact Riemann
surface, and using the result of \cite{NS} showed that the tensor
product of semistable filtered complex vector spaces is semistable. 
Totaro \cite{Tot} introduced a certain kind of metric called
\emph{good metric} on a filtered vector space over the field of
complex numbers and showed that the polystability of a filtered
vector space is equivalent to the existence of a good metric on the
vector space, and using this he gave an alternate proof of
Faltings' result.

The theory of quivers and their representations is important in the
study of representation theory.  In order to classify finite
dimensional modules over a finite dimensional algebra, King \cite{Kin}
translated this problem into classifying the representations of a
certain quiver and constructed the moduli spaces in question.  These
moduli spaces are closely related to the moduli spaces of vector
bundles (see for example, \cite{Hos}).

In this paper, we recall the notion of internal product $Q \otimes Q'$
of two quivers $Q$ and $Q'$ as defined by Li and Lin in \cite{LL}, and
by Keller in \cite{Kel}, and then we define the tensor product
$V \otimes W$ of a representation $V$ of $Q$ with a representation
$W$ of $Q'$.  We then show that the semistability of $V$ and $W$
implies the semistability of $V \otimes W$ over arbitrary fields.  As
a consequence we can get the similar result for geometrically stable
representations.  When the base field is $\C$ we give another proof of
the semistability of $V \otimes W$.  This alternate proof is similar
to the proof of semistability of tensor product of semistable filtered
vector spaces by Totaro.  In our case, the role of the good metric is
played by certain kind of Hermitian metrics on the representations
which we call \emph{Einstein-Hermitian metrics}.  An
Einstein-Hermitian metric on a representation of a quiver is a
solution of a certain differential equation called a \emph{quiver
vortex equation}.  The relation between stability and the solutions of
such equations has been studied by \'Alvarez-C\'onsul and Garc\'ia-Prada
in the general context of ``twisted quiver bundles" in \cite{CP}. 

A consequence of King's construction of the moduli spaces of quiver
representations is that they are quasi projective varieties.  In
particular, when the base field is that of the complex numbers, there
exists a positive Hermitian holomorphic line bundle on each such moduli
space.  In \cite{DMR, PD} the authors have constructed such a natural
line bundle on the moduli spaces of stable and semistable
representations of quivers respectively.  In this paper we give a
relation between such natural line bundles on the moduli spaces of
$Q, Q'$ and $Q \otimes Q'$ when the base field is $\C$.  Furthermore,
we also describe a relation between the universal families of
representations on the fine moduli spaces of representations of $Q, Q'$
and $Q \otimes Q'$ over arbitrary algebraically closed fields.  Note
that, here we are associating to an object of the product category
$\Repr[k]{Q} \times \Repr[k]{Q'}$ an object of the category
$\Repr[k]{Q \otimes Q'}$, where $\Repr[k]{Q}$ denotes the category of
representations of the quiver $Q$ over the field $k$.

The fixed loci of torus action on a quiver moduli space can be described
using the Bia\l{}ynicki-Birula decomposition and these connected
components can be  realised as moduli spaces of covering quivers, see
\cite{Wei, Fra}.  We compare the semi-stability under tensor product
construction with the (abelian) covering quiver construction and observe
that these two constructions do not commute in general.  Any
representation $V$ of a quiver lifts to a representation
$\widetilde{V}$ of the covering quiver and becomes (semi-) stable
w.r.t. a suitable slope function.  As a consequence we get a map
$\tau_N$ (see Proposition \ref{prop:4.4}) from the moduli space for the 
covering quiver $\tilde{Q}$ to the moduli space for the original quiver
$Q$.  We also find some relations between stability conditions of
$\widetilde{V} \otimes \widetilde{W}$ and $\widetilde{(V \otimes W)}$
where $V$ and $W$ are representations of the quivers $Q$ and $Q'$
respectively.  We further prove the descent of the natural line bundle
similar to the descent under the tensor product map $\bar{\phi}$.  We
combine these results with the map given by the tensor product to get
a commutative diagram (see Corollary \ref{cor:4.7}) and then use it to
give a relation between the natural line bundles on the moduli spaces of
representations of covering quivers and their internal product.

In Section \ref{sec:1}, we recall the basics on quivers and their
representations, and describe the moduli space of representations of
a quiver as constructed by King and mention some of its properties.
In Section \ref{sec:2}, we define the tensor product of two
representations and show that the tensor product of semistables is
semistable.  In Section \ref{sec:3}, we describe the relations between
the natural line bundles, and between the universal families of
representations on the fine moduli spaces of the quivers involved.  In
Section \ref{sec:4}, we recall the notion of (universal) covering quiver
of a quiver with a fixed dimension vector. We give the relation between
the internal product of quivers with the internal product of (universal)
covering quivers. This is used to get the relation between stability of
tensor product of representations on the (universal) covering quiver. 
We also deduce the description of the pull back of the natural line
bundle on the moduli of internal product of covering
quivers in terms of suitable powers of the natural line bundles on the
individual quiver moduli spaces.

\section{The moduli space of quiver representations}
	\label{sec:1}

\subsection{Quivers and their representations}
	\label{sec:1.1}

A \emph{quiver} is a quadruple $Q = (Q_0,Q_1,\sr,\tg)$ consisting of two
finite sets $Q_0 \neq \emptyset$ and $Q_1$ called the set of vertices
and the set of arrows respectively, and two functions
$\sr, \tg : Q_1\to Q_0$ called the source map and the target map
respectively.

Let $k$ be any field.

A \emph{representation of $Q$ over $k$} is a pair $(V,\rho)$, where
$V=(V_a)_{a\in Q_0}$ is a family of finite-dimensional $k$-vector
spaces and $\rho=(\rho_\alpha)_{\alpha\in Q_1}$ is a family of
$k$-linear maps
\begin{equation*}
\rho_\alpha : V_{\sr(\alpha)} \to V_{\tg(\alpha)}.
\end{equation*}
Let $(V,\rho)$ and $(W,\sigma)$ be two representations of $Q$.  A
\emph{morphism from $(V,\rho)$ to $(W,\sigma)$} is a family
$f=(f_a)_{a\in Q_0}$ of $k$-linear maps $f_a : V_a \to W_a$ satisfying
$\sigma_\alpha \circ f_{\sr(\alpha)} = f_{\tg(\alpha)} \circ
\rho_\alpha$ for all $\alpha\in Q_1$.  The morphism $f$ is an
\emph{isomorphism} if each $f_a$ is an isomorphism.
 
\subsection{Stability and family of representations}
	\label{sec:1.2}
	
Let $(V,\rho)$ be a representation of $Q$.  The \emph{type} or the
\emph{dimension vector} of $(V,\rho)$ is the element
$(\dim[k]{V_a})_{a\in Q_0}$ of $\N^{Q_0}$, denoted $\dim{V,\rho}$.
The \emph{rank} of $(V,\rho)$ is the natural number 
$\sum_{a\in Q_0}\dim[k]{V_a}$, denoted $\rk{V,\rho}$.
A \emph{weight} of $Q$ is an element $\theta$ of the $\R$-vector
space $\R^{Q_0}$.
The \emph{$\theta$-degree} of $(V,\rho)$ is the real number
\begin{equation*}
   \label{eq:70}
\deg_\theta(V,\rho) = \sum_{a\in Q_0}\theta_a \dim[k]{V_a}.
\end{equation*}
If $(V,\rho)$ is non-zero, the \emph{$\theta$-slope} of $(V,\rho)$ is
the real number
\begin{equation*}
   \label{eq:71}
\mu_\theta(V,\rho) = \frac{\deg_\theta(V,\rho)}{\rk{V,\rho}}.
\end{equation*}
More generally, for any $d \in \N^{Q_0}\setminus \set{0}$, we define
$\rk{d} = \sum_{a \in Q_0} d_a,
\deg_\theta(d) = \sum_{a \in Q_0}\theta_a d_a$
and $\mu_\theta(d) = \frac{\deg_\theta(d)}{\rk{d}}$.

A non-zero representation $(V,\rho)$ is called
\emph{$\theta$-semistable} (respectively, \emph{$\theta$-stable}) if
\begin{equation*}
   \label{eq:72}
\mu_\theta(W,\sigma) \leq \mu_\theta(V,\rho) \quad
(\text{respectively,}~ \mu_\theta(W,\sigma) < \mu_\theta(V,\rho))
\end{equation*}
for every non-zero proper subrepresentation $(W,\sigma)$ of $(V,\rho)$.
The representation $(V,\rho)$ is called \emph{$\theta$-polystable} if
it is $\theta$-semistable, and if it is a direct sum of a finite family
of $\theta$-stable representations of $Q$ each having slope same as that
of $(V,\rho)$.


Let $d=(d_a)_{a \in Q_0}$ be an element of $\N^{Q_0}$ such that
$d_a > 0$ for all $a \in Q_0$.  A \emph{family} of $\theta$-semistable
representations of type $d$ parametrised by a $k$-scheme $T$ is a pair
$(V,\rho)$, where $V=(V_a)_{a\in Q_0}$ is a family of locally free
$\struct{T}$-modules, $\rho=(\rho_\alpha)_{\alpha\in Q_1}$ is a family
of morphisms $\rho_\alpha : V_{\sr(\alpha)} \to V_{\tg(\alpha)}$ of
$\struct{T}$-modules and for every closed point $t$ of $T$, the
representation $(V(t),\rho(t))$ of $Q$ over $k$ is $\theta$-semistable
and has type $d$.

We have an obvious notion of \emph{morphism} between two families
parametrised by $T$ which generalises the morphism of representations 
of a quiver as defined earlier.  Two families $(V,\rho)$ and
$(W,\sigma)$ parametrised by $T$ are called \emph{equivalent} if there
is an invertible $\struct{T}$-module $L$ such that
$(L\otimes_{\struct{T}} V,\id{L}\otimes \rho)$ is isomorphic to
$(W,\sigma)$.

\subsection{The moduli of $\theta$-semistable representations of type $d$} 
	\label{sec:1.3}
Let $\theta$ be a rational weight of $Q$ and $d \in \N^{Q_0}$ be fixed,
and let $k$ be an algebraically closed field.  The
\emph{moduli functor of type $(\theta,d)$} of $Q$ is the presheaf
$F(\theta,d)$ on the category of $k$-schemes, which is defined by
\begin{align*}
F(\theta,d) (T) = \;
&\text{set of equivalence classes of families of $\theta$-semistable}\\
&\text{representations of type $d$ parametrised by $T$}.
\end{align*}

\begin{proposition}\cite[Proposition 5.2]{Kin}
	\label{prop:king}
There exists a connected normal quasi-projective $k$-scheme
$M(\theta,d)$, which universally corepresents the functor
$F(\theta,d)$.
\end{proposition}

The moduli space $M(\theta,d)$ has the following properties.
\begin{enumerate}
\item The moduli space $M(\theta,d)$ is projective if and only if $Q$
is acyclic. \label{item:1.3.1}

\item There is a canonical bijection from the set of S-equivalence
classes of $\theta$-semistable representations of $Q$, to the set
$N(\theta,d)$ of closed points of $M(\theta,d)$, where we say that two
representations are S-equivalent if they have filtrations with stable
quotients and the same associated graded representations. \label{item:1.3.2}

\item There is an open subset $\stab{M}(\theta,d)$ of $M(\theta,d)$
such that the set
\begin{equation*}
   \label{eq:6}
\stab{N}(\theta,d) = N(\theta,d) \cap \stab{M}(\theta,d)
\end{equation*}
of its closed points is the image of the set of isomorphism classes of
$\theta$-stable representations of type $d$ of $Q$, under the above
bijection. \label{item:1.3.3}

\item If $d$ is coprime, that is,
$\gcd(d_a \suchthat a \in Q_0) = 1$, 
then the moduli space $\stab{M}(\theta,d)$ of $\theta$-stable 
representations is a fine moduli space.  Therefore, there exists a 
universal family of representations $(U,\psi)$ parametrised by
$\stab{M}(\theta,d)$ (see \cite[Proposition 5.3]{Kin}). \label{item:1.3.4}

\item If $d$ is $\theta$-coprime, that is, $\mu_\theta (d') \neq
\mu_\theta (d)$ for all $0 < d' < d$, then $M(\theta,d) =
\stab{M}(\theta,d)$. \label{item:1.3.5}

\item If $d$ is coprime, then there exists a dense open subset of
$\R^{Q_0}$ such that for any weight $\theta$ in that subset $d$ is also
$\theta$-coprime. \label{item:1.3.6}

\end{enumerate}

\subsection{A description of the moduli space}
	\label{sec:1.4}

Fix a family $V=(V_a)_{a\in Q_0}$ of $k$-vector spaces, such that
$\dim[k]{V_a}=d_a$ for all $a\in Q_0$.  Let $\mathcal{A}$ denote the
\emph{representation space} of type $d$ of $Q$, that is, the
finite-dimensional $k$-vector space
$\bigoplus_{\alpha\in
  Q_1}\Hom[k]{V_{\sr(\alpha)}}{V_{\tg(\alpha)}}$.
We give the vector space $\mathcal{A}$ its usual structure of a variety.

Let $G$ denote the algebraic group $\prod_{a\in Q_0}\Aut[k]{V_a}$.
There is a canonical algebraic right action of $G$ on $\mathcal{A}$,
which is defined by
\begin{equation*}
  \label{eq:10}
  (\rho g)_\alpha =
  g_{\tg(\alpha)}^{-1} \circ \rho_\alpha \circ g_{\sr(\alpha)}
\end{equation*}
for all $\rho \in \mathcal{A}$, $g\in G$, and $\alpha\in Q_1$.  The
stabiliser of any point $\rho$ of $\mathcal{A}$ is canonically
identified with the automorphism group of the representation $(V,\rho)$
of $Q$.  The orbits of this action is in bijection with the set of
isomorphism classes of representations of $Q$ of type $d$.

Denote by $\Delta$ the central algebraic subgroup of $G$
consisting of all elements of the form $ce$, as $c$ runs over
$\units{k}$, where $e=(\id{V_a})_{a\in Q_0}$ is the identity element
of $G$.  Let $\bar{G}$ denote the algebraic group
$\Delta \backslash G$, and $\pi : G \to \bar{G}$ the canonical
projection.  Then, the action of $G$ on $\mathcal{A}$ induces an
algebraic right action of $\bar{G}$ on $\mathcal{A}$.

Let $\sstab{\mathcal{A}}$ (respectively, $\stab{\mathcal{A}}$) denote
the set of points $\rho$ in $\mathcal{A}$, such that the
representation $(V,\rho)$ of $Q$ is $\theta$-semistable (respectively,
$\theta$-stable).  There is an obvious family of $\theta$-semistable
representations of type $d$ of $Q$ parametrised by
$\sstab{\mathcal{A}}$.  Therefore, the universal property of
$M(\theta,d)$ gives a morphism of varieties
$\sstab{p} : \sstab{\mathcal{A}} \to M(\theta,d)$.  
\begin{proposition}[{\cite[Proposition~5.2]{Kin}}]
	\label{prop:1.1}
The morphism $\sstab{p}$ is a
good quotient of $\sstab{\mathcal{A}}$ by $\bar{G}$, and its
restriction $\stab{p} : \stab{\mathcal{A}} \to \stab{M}(\theta,d)$ is
a geometric quotient of $\stab{\mathcal{A}}$ by $\bar{G}$.
\end{proposition}

\subsection{Moduli space over the field of complex numbers}
	\label{sec:1.5}
Let $k = \C$. We then have an equivalent
statement for the polystability which we describe now.  A
\emph{Hermitian metric} on a representation $(V,\rho)$ of $Q$ is a
family $h=(h_a)_{a\in Q_0}$ of Hermitian inner products
$h_a : V_a\times V_a \to \C$.
We say that the metric $h$ is \emph{Einstein-Hermitian} with respect
to $\theta$ if it satisfies
\begin{equation*}
  \label{eq:75}
  \sum_{\alpha\in \tg^{-1}(a)} \rho_\alpha \circ \rho_\alpha^* -
  \sum_{\alpha\in \sr^{-1}(a)}\rho_\alpha^* \circ \rho_\alpha =
  (\mu_\theta(V,\rho) - \theta_a) \; \id{V_a}
\end{equation*}
for all $a\in Q_0$, where, for each $\alpha\in Q_1$,
$\rho_\alpha^* : V_{\tg(\alpha)} \to V_{\sr(\alpha)}$ is the adjoint
of $\rho_\alpha : V_{\sr(\alpha)} \to V_{\tg(\alpha)}$ with respect to
the Hermitian inner products $h_{\sr(\alpha)}$ and $h_{\tg(\alpha)}$
on $V_{\sr(\alpha)}$ and $V_{\tg(\alpha)}$, respectively.

We then have the following proposition which is essentially a
consequence of \cite[Proposition 6.5]{Kin}.

\begin{proposition}\cite[Proposition 4.3]{DMR}
  \label{pro:1.2}
  Let $\theta$ be a rational weight of $Q$, and $(V,\rho)$ a non-zero
  representation of $Q$.  Then, $(V,\rho)$ has an Einstein-Hermitian
  metric with respect to $\theta$ if and only if it is
  $\theta$-polystable.  Moreover, if $h_1$ and $h_2$ are two
  Einstein-Hermitian metrics on $(V,\rho)$ with respect to $\theta$,
  then there exists an automorphism $f$ of $(V,\rho)$, such that
  \begin{equation*}
    \label{eq:79}
    h_{1,a}(v,w) = h_{2,a}(f_a(v),f_a(w))
  \end{equation*}
  for all $a\in Q_0$ and $v,w\in V_a$.
\end{proposition}

Moreover, in addition to the properties (1)--(6) as stated in Section
\ref{sec:1.3}, the moduli spaces of quiver representations have the
following extra properties.

\begin{enumerate}
\item[(7)] The moduli space $N(\theta,d)$ of semistable representations
is a connected, normal and quasi-projective complex analytic space, and
$\stab{N}(\theta,d)$ is a dense open subspace in $N(\theta,d)$ and has
a natural structure of a symplectic manifold (see
\cite[Section 6]{Kin}).\label{item:1.3.7}

%
%

\item[(8)] There exist a natural K\"ahler metric $\stab{g}$, and a
natural  Hermitian holomorphic line bundle
$(\stab{L},\stab{h})$ on $\stab{N}(\theta,d)$ such that
$c_1(\stab{L},\stab{h}) = \frac{n}{2\pi} \stab{\Theta}$, where
$n$ is a positive integer such that $n(\theta_a-\mu_\theta(d))\in \Z$ 
for all $a\in Q_0$, and $\stab{\Theta}$ is the K\"ahler form of
$\stab{g}$ (see \cite[Theorem 7.8, 8.3]{DMR}). \label{item:1.3.10}

\item[(9)] There exist a natural holomorphic line bundle $L$ on
$N(\theta, d)$, a natural continuous Hermitian metric $h$ on $L$ and
a natural K\"ahler stratification (see \cite[Section 4c]{PD} for
definition) of $N(\theta, d)$ such that for each stratum $Z$, the
Hermitian metric $h_Z = \restrict{h}{Z}$ on $L_Z = \restrict{L}{Z}$ is
smooth and $c_1(L_Z, h_Z) = \frac{n}{2\pi}\Theta_Z$, where $\Theta_Z$
is the K\"ahler form on $Z$ (see \cite[Theorem 5.3, 5.4]{PD}). \label{item:1.3.11}
\end{enumerate}

\begin{remark}
	\label{rem:1.1}
By a theorem of King \cite[Theorem 6.1, 6.5]{Kin} we also have that a
point $\rho \in \mathcal{A}$ belongs to $\sstab{\mathcal{A}}$ if and
only if the closure, in the strong topology of $\mathcal{A}$, of the
$G$-orbit $\rho G$ contains a point $\rho' \in \mathcal{A}$ such that
the representation $(V,\rho')$ is polystable.
\end{remark}

\section{Tensor products of representations and their semistability}
	\label{sec:2}
	

Let $Q = (Q_0,Q_1,\sr_Q,\tg_Q)$ and $Q' = (Q_0',Q_1',\sr_{Q'},\tg_{Q'})$
be two finite quivers.
\begin{definition}
The internal product of $Q$ and $Q'$, denoted $Q \otimes Q'$, is a
quiver with
\begin{center}
 $(Q\otimes Q')_0 = Q_0\times Q_0'$,  \qquad $(Q\otimes Q')_1 =
 (Q_0\times Q_1')\sqcup (Q_1\times Q_0')$,
\end{center}
and whose source and target maps are defined by
\begin{align*}
\sr_{Q\otimes Q'}((a,\alpha')) &= (a,\sr_{Q'}(\alpha')), \quad
 \tg_{Q\otimes Q'}((a,\alpha')) &&= (a,\tg_{Q'}(\alpha')) \quad
 &\mathrm{for}~(a,\alpha') \in Q_0\times Q_1'\\
\sr_{Q\otimes Q'}((\alpha,a')) &= (\sr_Q(\alpha),a'), \quad
 \tg_{Q\otimes Q'}((\alpha,a')) &&= (\tg_Q(\alpha),a') \quad
 &\mathrm{~for}~(\alpha,a') \in Q_1\times Q_0'.
\end{align*}
\end{definition}

\begin{example}
Let $Q$ be the quiver $A_n$ and $Q'$ be the Kronecker $1$-quiver.
\begin{equation*}
\begin{tikzcd}
1 \arrow[r] & 2 \arrow[r] & \cdots \arrow[r] & n
\end{tikzcd}
\qquad
\begin{tikzcd}
a \arrow[r] & b
\end{tikzcd}
\end{equation*}
Then their internal product is the following quiver-
\begin{equation*}
\begin{tikzcd}
(1,a) \arrow[r] \arrow[d] & (2,a) \arrow[r] \arrow[d] & \cdots \arrow[r] & (n,a) \arrow[d] \\
(1,b) \arrow[r] & (2,b) \arrow[r] & \cdots \arrow[r] & (n,b)
\end{tikzcd}.
\end{equation*}

\end{example}

%
%

\begin{remark}
\begin{enumerate}
\item The above product of quivers was defined by Li and Lin in
\cite[Section 3.2]{LL} where they called this product the ``product
valued quiver", and by Keller in \cite[Section 3.3]{Kel} where he
called this product the ``tensor product" of the quivers.\\

\item For each vertex $a\in Q_0$, the full-subquiver on the vertex set
$\{a\}\times Q_0'$ is isomorphic to the quiver $Q'$, and for each
vertex $a' \in Q_0'$, the full-subquiver on the vertex set
$Q_0\times \{a'\}$ is isomorphic to the quiver $Q$.
\end{enumerate}
\end{remark}


\subsection{Semistability of tensor product of representations over arbitrary fields}
	\label{sec:2.3}

Let $k$ be any field.

Let $\theta = (\theta_a)_{a \in Q_0}$ and
$\theta' = (\theta'_{a'})_{a'\in Q_0'}$ be rational weights of $Q$ and
$Q'$ respectively.  Let  $\overline{\theta} =
(\theta_a + \theta'_{a'})_{(a,a')\in Q_0\times Q_0'}$.

Let $(V,\rho) =
\left((V_a)_{a \in Q_0}, (\rho_{\alpha})_{\alpha\in Q_1}\right)$ and
$(W,\sigma) = \left((W_{a'})_{a'\in Q_0'},
(\sigma_{\alpha'})_{\alpha'\in Q_1'}\right)$ be representations of $Q$
and $Q'$ over $k$ respectively.
\begin{definition}
The tensor product of $(V,\rho)$ and $(W,\sigma)$ is the
representation
$$(V,\rho) \otimes (W,\sigma) = ((V_a\otimes W_{a'})_{(a,a') \in Q_0\times Q_0'}, (\tau_{\gamma})_{\gamma \in (Q\otimes Q')_1}))$$
of the quiver $Q\otimes Q'$, where
\begin{equation*}
 \tau_\gamma =
\begin{cases}
\id{V_a}\otimes \sigma_{\alpha'} \quad 
 \mbox{~for} \quad \gamma = (a,\alpha') \in Q_0\times Q_1' \\
\rho_{\alpha} \otimes \id{W_{a'}} \quad
 \mbox{for} \quad \gamma = (\alpha,a') \in Q_1\times Q_0'.
\end{cases} 
\end{equation*} 
\end{definition}


Let $\dim{V,\rho} = d = (d_a)_{a \in Q_0}$, and
$\dim{W,\sigma} = d' = (d_{a'}')_{a' \in Q_0'}$.  Then we have
\begin{align*}
 &\dim{(V,\rho) \otimes (W,\sigma)} = \bar{d} = (d_a\, d_{a'}')_{(a,a') \in Q_0\times Q_0'},\\
 &\rk{(V,\rho) \otimes (W,\sigma)} = \sum_{a\in Q_0, a' \in Q_0'} d_a\, d'_{a'} =
 \rk{V,\rho}\, \rk{W,\sigma},\\
 &\deg_{\overline{\theta}}{((V,\rho) \otimes (W,\sigma))} = \deg_{\theta}{(V,\rho)}\,
 \rk{W,\sigma} + \deg_{\theta'}{(W,\sigma)}\, \rk{V,\rho},
\end{align*}
and therefore,
\begin{equation*}
 \mu_{\overline{\theta}}((V,\rho) \otimes (W,\sigma)) = \mu_{\theta}(V,\rho) +
 \mu_{\theta'}(W,\sigma).
\end{equation*}


\begin{proposition} \label{prop:2.1}
Suppose that the representation $(V,\rho)$ of $Q$ is
$\theta$-semistable and the representation $(W,\sigma)$ of $Q'$is
$\theta'$-semistable.  Then the representation
$(U,\tau) = (V,\rho) \otimes (W,\sigma)$ of $Q\otimes Q'$ is
$\overline{\theta}$-semistable.
\end{proposition}

\begin{proof}
Let $(V, \rho)$ be a $\theta$-semistable representation of the quiver
$Q$ and $(W, \sigma)$ be a $\theta'$-semistable representation of the
quiver $Q'$.  Suppose $(V, \rho) \otimes (W, \sigma) $ is not
$\overline{\theta}$-semistable.  Then we get a quotient
$(V, \rho) \otimes (W, \sigma) \rightarrow (U, \tau) $ such that
$\mu_{\overline{\theta}} ((V, \rho) \otimes (W, \sigma) ) > \mu_{\overline{\theta}}(U, \tau)$.
 
By multiplying $\theta$ and $\theta'$ by sufficiently large integer
$d$ we can assume that 
$$\mu_{\overline{\theta}} ((V, \rho) \otimes (W, \sigma) ) - \mu_{\overline{\theta}}(U, \tau) =   \mu_{\theta}(V,\rho) +
 \mu_{\theta'}(W,\sigma) - \mu_{\overline{\theta}}(U,\tau) > 2,$$ 
as scaling does not change the semistability, see
\cite[Section 5.1]{Re08}.
 
Now we can find the integers $m$ and $n$ such that 
$$\mu_{\theta} (V, \rho) + m > 0, ~ \mu_{\theta'} (W, \sigma) + n > 0 
\mbox{ and } m + n + \mu_{\overline{\theta}}(U, \tau) \leq 0.$$
 
If we replace the weight $\theta$ by $\theta_1 :=
(\theta_a + m)_{a \in Q_0} $ and $\theta'$ by $\theta_1' :=
(\theta_{a'} + n)_{a' \in Q_0'} $, then the semistability does not
change following \cite[Section 5.1]{Re08}.  Hence $(V, \rho)$
(respectively $(W, \sigma)$) is $\theta_1$-semistable (respectively
$\theta_1'$-semistable) such that 
$$\mu_{\theta_1} (V, \rho) > 0, ~ \mu_{\theta_1'} (W, \sigma) > 0 \mbox{ and } \mu_{\overline{\theta_1}}(U, \tau) \leq 0 $$
are satisfied.  In particular 
$$\deg_{\theta_1} (V, \rho) > 0, ~ \deg_{\theta_1'} (W, \sigma) > 0 \mbox{ and } \deg_{\overline{\theta_1}}(U, \tau) \leq 0$$
holds by multiplying the ranks.

Now using the lemma \ref{lemma:2.5} we get that
$\deg_{\overline{\theta_1}}(U, \tau) > 0$ which is a contradiction.
Hence $(V, \rho) \otimes (W, \sigma)$ is $\overline{\theta}$-semistable.
\end{proof}
 
\begin{lemma} \label{lemma:2.5}
	\begin{enumerate}
		\item[(i)] A direct sum of semistable representations of
		positive degree of same slope is again a semistable
		representation of positive degree.
		
		\item[(ii)] Any quotient of positive degree semistable
		representation is again a representation of positive degree.
		
		\item[(iii)] For each vertex $(a, a') \in Q_0 \times Q_0'$,
		there exist restriction functors
		$$ R_{a, Q'} : \Repr[k]{Q \otimes Q'} \rightarrow \Repr[k]{Q'};
		(U, \tau) \mapsto ((U_{(a,a')})_{a' \in Q_0'},
		(\tau_{(a, \alpha')})_{\alpha' \in Q_1'} )$$
		$$ \mbox{ and } R_{Q, a'} : \Repr[k]{Q \otimes Q'} \rightarrow
		\Repr[k]{Q} ; (U, \tau) \mapsto ((U_{(a,a')})_{a \in Q_0},
		(\tau_{( \alpha, a')})_{\alpha \in Q_1} )$$ 
		which are exact functors of abelian categories.
		\item[(iv)] For each vertices $a \in Q_0$, the restriction
		$R_{a, Q'} ((V, \rho) \otimes (W, \sigma)) =
		((V_a \otimes W_{a'})_{a' \in Q_0'},
		(\id{V_a} \otimes \sigma_{\alpha'})_{\alpha' \in Q'_1})
		\simeq (W, \sigma)^{\oplus d_a}$.  Similarly, we have
		$R_{Q, a'} ((V, \rho) \otimes (W, \sigma)) \simeq
		(V, \rho)^{\oplus d'_{a'}}$.
		
		\item[(v)] If $(U, \tau)$ is a quotient of
		$(V, \rho) \otimes (W, \sigma)$ which are semistable
		representations of positive degree then
		$\deg_{\overline{\theta_1}}(U, \tau) > 0$.
	\end{enumerate}
\end{lemma}
\proof
The proof of the first four assertions easily follows from the
definition.  Now to prove the last assertion we observe that 
 \begin{eqnarray}
 \deg_{\overline{\theta_1}}(U, \tau) &=& \sum_{(a, a') \in Q_0
 \times Q_0'} (\theta_{1, a} + \theta'_{1, a'}) \dim{U_{(a, a')}} \nonumber\\
  &=& \sum_{ a' \in Q_0'}  \Big[\sum_{a \in Q_0}
  (\theta_{1, a} \dim{U_{(a, a')}} \Big]+ \sum_{ a \in Q_0}
  \Big[\sum_{a' \in Q_0'} \theta'_{1, a'}) \dim{U_{(a, a')}}\Big] \nonumber \\
  &=& \sum_{ a' \in Q_0'} [\deg_{\theta_1} (R_{Q, a'}(U, \tau))] +
  \sum_{a \in Q_0} [\deg_{\theta_1'} (R_{a, Q'}(U, \tau))]  \nonumber \\
  &>& 0 ~(\mbox{ using } (i), (ii) \mbox{ and } (iv)). \nonumber \hspace*{6.5cm} \qed
 \end{eqnarray}
 
Recall that a representation $(V,\rho)$ of a quiver $Q$ over a field
$k$ is said to be \emph{$\theta$-geometrically stable} if for any
field extension $k'$ of $k$, the representation
$(V,\rho)\otimes k' = ((V_a \otimes_k k')_{a \in Q_0},(\rho_\alpha \otimes_k \id{k'})_{\alpha \in Q_1})$ is $\theta$-stable (see
\cite[Definition 2.19]{Hos}).  Since for any two $k$-vector spaces
$V$ and $W$ we have
\begin{equation*}
(V \otimes_k W) \otimes_k k' \cong (V \otimes_k k' \otimes_{k'} W) \otimes_k k' \cong (V \otimes_k k') \otimes_{k'} (W \otimes_k k'),
\end{equation*}
we get the following corollary to Proposition \ref{prop:2.1}.

\begin{corollary}
	\label{cor:2.6}
If the representation $(V,\rho)$ of $Q$ is $\theta$-geometrically
stable and the representation $(W,\sigma)$ of $Q'$ is
$\theta'$-geometrically stable with types $d$ being
$\theta$-coprime, $d'$ being $\theta'$-coprime and $dd'$ being
$\bar{\theta}$-coprime, then the representation
$(V,\rho) \otimes (W,\sigma)$ of $Q \otimes Q'$ is
$\bar{\theta}$-geometrically stable.
\end{corollary}

\subsection{Semistability of tensor product of representations over the field of complex numbers}
	\label{sec:2.2}

Assuming $k = \C$ we give another proof of Proposition \ref{prop:2.1}. To this end, we first establish that the tensor product of polystables
is polystable. 
\begin{lemma}
\label{lemma:2.1}
Suppose that the representation $(V,\rho)$ of $Q$ is
$\theta$-polystable and the representation $(W,\sigma)$ of $Q'$is
$\theta'$-polystable.  Then, the representation $(V,\rho) \otimes
(W,\sigma)$ of $Q\otimes Q'$ is $\overline{\theta}$-polystable.
\end{lemma}
\proof For a quiver $Q$ and a representation $(V,\rho)$ of $Q$, let us
write 
\begin{equation*}
X^Q_a(V,\rho) =
\sum_{\alpha \in \sr^{-1}(a)} \rho_{\alpha}^*\circ \rho_{\alpha}
\qquad \mbox{and} \qquad 
Y^Q_a(V,\rho) = \sum_{\alpha \in \tg^{-1}(a)}
\rho_{\alpha}\circ \rho_{\alpha}^*.
\end{equation*}
Then by Proposition \ref{pro:1.2}
we have
\begin{align*}
 Y^Q_a(V,\rho) - X^Q_a(V,\rho) 
  &=(\mu_{\theta}(V,\rho) - \theta_a)\id{V_a}
  &(a \in Q_0)\\
 Y^{Q'}_{a'}(W,\sigma) - X^{Q'}_{a'}(W,\sigma)
  &=(\mu_{\theta'}(W,\sigma) - \theta'_{a'})\id{W_{a'}}
  &(a' \in Q_0'),
\end{align*}
with respect to some Hermitian metrics $h$ on $(V,\rho)$ and $h'$ on
$(W,\sigma)$.  Note that, $\bar{\sr}^{-1}(a,a') =
\{a\} \times \sr'^{-1}(a')\sqcup \sr^{-1}(a)\times \{a'\}$.
Therefore,
\begin{align*}
X^{Q\otimes Q'}_{(a,a')}((V,\rho) \otimes (W,\sigma)) &= \sum_{\alpha' \in \sr'^{-1}(a')}
 (\id{V_a}\otimes \sigma_{\alpha'})^* \circ (\id{V_a}\otimes
 \sigma_{\alpha'}) + \sum_{\alpha \in \sr^{-1}(a)}
 (\rho_{\alpha}\otimes \id{W_{a'}})^* \circ (\rho_{\alpha}\otimes
 \id{W_{a'}})\\
&= \id{V_a}\otimes(\sum_{\alpha' \in \sr'^{-1}(a')} \sigma_{\alpha'}^*
 \circ \sigma_{\alpha'}) + (\sum_{\alpha \in \sr^{-1}(a)}
 \rho_{\alpha}^* \circ \rho_{\alpha}) \otimes \id{W_{a'}} \\
&= \id{V_a}\otimes X^{Q'}_{a'}(W,\sigma) + X^Q_a(V,\rho) \otimes
\id{W_{a'}}.
\end{align*}
Similarly,
\begin{equation*}
Y^{Q\otimes Q'}_{(a,a')}((V,\rho) \otimes (W,\sigma)) = \id{V_a}\otimes Y^{Q'}_{a'}(W,\sigma) + Y^Q_a(V,\rho) \otimes \id{W_{a'}}.
\end{equation*}
Therefore, with respect to the Hermitian inner product
$h_a \otimes h_{a'}'$ on $V_a \otimes W_{a'}$ we have
\begin{align*}
&Y^{Q\otimes Q'}_{(a,a')}((V,\rho) \otimes (W,\sigma)) - X^{Q\otimes Q'}_{(a,a')}((V,\rho) \otimes (W,\sigma))\\
&= \id{V_a}\otimes (Y^{Q'}_{a'}(W,\sigma) - X^{Q'}_{a'}(W,\sigma)) +
(Y^Q_a(V,\rho) - X^Q_a(V,\rho))\otimes \id{W_{a'}}\\
&= \id{V_a}\otimes (\mu_{\theta'}(W,\sigma) - \theta'_{a'})\id{W_{a'}}
+ (\mu_{\theta}(V,\rho) - \theta_a)\id{V_a} \otimes \id{W_{a'}}\\
&= (\mu_{\overline{\theta}}((V,\rho) \otimes (W,\sigma)) - (\overline{\theta})_{(a,a')})
\id{V_a}\otimes \id{W_{a'}}.
\end{align*}
Hence, the representation $(V,\rho) \otimes (W,\sigma)$ of $Q\otimes Q'$
is $\overline{\theta}$-polystable. \qed


\proof[Alternate proof of Proposition \ref{prop:2.1} when $k=\C$] Let
$$\mathcal{A}_1 = \mathcal{A}(Q,d) =
\bigoplus_{\alpha\in Q_1}\Hom[\C]{V_{\sr(\alpha)}}{V_{\tg(\alpha)}}$$
be the representation space of $Q$ of type $d$, and
$$\mathcal{A}_2 = \mathcal{A}(Q',d') =
\bigoplus_{\alpha'\in Q_1'}\Hom[\C]{W_{\sr'(\alpha')}}
{W_{\tg'(\alpha')}}$$
be the representation space of $Q'$ of type $d'$.
Then,
\begin{small}
\begin{equation*}
\mathcal{A}_3 = \Big(\bigoplus_{\alpha \in Q_1}\bigoplus_{a' \in Q_0'}
\Hom[\C]{V_{\sr(\alpha)} \otimes W_{a'}}{V_{\tg(\alpha)}
\otimes W_{a'}}\Big)
\bigoplus
\Big(\bigoplus_{a \in Q_0} \bigoplus_{\alpha' \in Q_1'}
\Hom[\C]{V_a \otimes W_{\sr'(\alpha')}}
{V_a \otimes W_{\tg'(\alpha')}}\Big) 
\end{equation*}
\end{small}
is the representation space of $Q\otimes Q'$
of type $\bar{d}$.  Let $G_1 = \prod_{a\in Q_0}\Aut[\C]{V_a}$ and
$G_2 = \prod_{a'\in Q_0'}\Aut[\C]{W_{a'}}$.
Since $(V,\rho)$ is  $\theta$-semistable and $(W,\sigma)$ is
$\theta'$-semistable, by Remark \ref{rem:1.1} $\bar{\rho G_1}$ contains
a point $\rho' \in \mathcal{A}_1$ such that the representation
$(V,\rho')$ of $Q$ is $\theta$-polystable, and $\bar{\sigma G_2}$
contains a point $\sigma' \in \mathcal{A}_2$ such that the
representation $(W,\sigma')$ of $Q'$ is $\theta'$-polystable.  Since
$\mathcal{A}_1$ and $\mathcal{A}_2$ are normed vector spaces, there
exist sequences $\{g_n\}$ in $G_1$ and $\{g_n'\}$ in $G_2$ such that
$\rho g_n \rightarrow \rho'$ and $\sigma g_n'\rightarrow \sigma'$.  Let
$k_{n,(a,a')} = g_{n,a}\otimes g_{n,a'}'$.  Then,
$k_n = {(k_{n,(a,a')})}_{(a,a') \in Q_0\times Q_0'} \in 
\prod_{(a,a') \in Q_0\times Q_0'} \Aut{V_a \otimes W_{a'}} = G_3$, and
the representation $(U,\tau') = (V,\rho') \otimes (W,\sigma')$ of
$Q\otimes Q'$ is $\overline{\theta}$-polystable by Proposition
\ref{lemma:2.1}.  Now, we have
\begin{align*}
(\tau k_n)_{(\alpha,a')} &= k_{n,(\tg(\alpha),a')}^{-1}\circ
\tau_{(\alpha,a')} \circ k_{n,(\sr(\alpha),a')} \\
&= (g_{n, \tg(\alpha)}^{-1} \otimes {g'}_{n,a'}^{-1}) \circ
(\rho_{\alpha} \otimes \id{W_{a'}}) \circ (g_{n, \sr(\alpha)} \otimes
{g'}_{n,a'})\\
&= (g_{n, \tg(\alpha)}^{-1} \circ \rho_{\alpha} \circ
g_{n, \sr(\alpha)}) \otimes \id{W_{a'}} = (\rho g_n)_{\alpha} \otimes
\id{W_{a'}}. 
\end{align*} 
Therefore $(\tau k_n)_{(\alpha,a')} \rightarrow \rho_{\alpha}' \otimes
\id{W_{a'}} = (\rho' \otimes \sigma')_{(\alpha,a')} =
\tau_{(\alpha,a')}'$.  Similarly, $(\tau k_n)_{(a,\alpha')}
\rightarrow \tau_{(a,\alpha')}'$.  Hence, $(U,\tau)$ is semistable.
\qed

\begin{remark} \label{rem:2.9}
Another way of establishing semistability is by using semi-invariants.
Let $\chi_1$ (respectively $\chi_2$) be the character of the group
$G_1$ (respectively $G_2$) defined by $\theta$ (respectively $\theta'$)
(see \cite[Section 3]{Re08}).  We denote the character of the group
$G_3$ defined by
$\bar{\theta}
$
by $\chi_3$.  Let $V$ (respectively $W$) be a $\chi_1$-semistable
(respectively $\chi_2$-semistable) point.  Then there is a
$\chi_1^{n_1}$-semi-invariant (respectively
$\chi_2^{n_2}$-semi-invariant) regular function $f_v$ (respectively
$f_w$) with $f_v(V) \neq 0$ (respectively $f_w(W) \neq 0$), where $n_1$
(respectively $n_2$) is a positive integer.
If we can construct a
semi-invariant regular function using these two functions which does
not vanish at $V\otimes W$, then we have the semistability of
$V\otimes W$.  The following discussion has the \emph{partial} answer to it.
\end{remark}


The map (see Equation \eqref{eq:phi} in Section \ref{sec:3.2})
$$\phi = (x_1,x_2) \mapsto x_1\otimes x_2 : \mathcal{A}_1 \times \mathcal{A}_2 \to \mathcal{A}_3$$
is a $u$-equivariant map, where
$$u : G_1 \times G_2 \to G_3$$
is the group homomorphism defined by
$$u(g_1,g_2)_{(a,a') \in Q_0 \times Q_0'} = (g_{3,(a,a')})_{(a,a') \in Q_0 \times Q_0'} = (g_{1,a}\otimes g_{2,a'})_{(a,a') \in Q_0 \times Q_0'}.$$
Therefore we have
$$x_1 g_1\otimes x_2 g_2 = \phi(x_1 g_1, x_2 g_2) = \phi\left((x_1,x_2) (g_1,g_2)\right) = \phi(x_1,x_2) u(g_1,g_2) = (x_1\otimes x_2) g_3.$$

Recall that for the $\chi$-semistability, the characters $\chi_1, \chi_2$ and
$\chi_3$ associated to the weights $\theta$, $\theta'$ and $\bar{\theta}$
respectively, are given by
\begin{align*}
\chi_1(g_1) &= \prod_{a\in Q_0} \det(g_{1,a})^{\deg_\theta(d) - \theta_a\, \rk{d}}\\
\chi_2(g_2) &= \prod_{a'\in Q_0'} \det(g_{2,a'})^{\deg_{\theta'}(d') - \theta_{a'}'\, \rk{d'}}\\
\chi_3(g_3)
&= \prod_{(a,a') \in Q_0 \times Q_0'} \det(g_{3,(a,a')})^{\deg_{\bar{\theta}}(\bar{d}) - \bar{\theta}_{a,a'}\, \rk{\bar{d}}}.
\end{align*}
%
These characters are related as follows.
\begin{small}
\begin{align*}
\chi_3(u(g_1,g_2)) &= \prod_{(a,a') \in Q_0 \times Q_0'}\left[\det(g_{1,a})^{d_{a'}'}\det(g_{2,a'})^{d_a}\right]^{\deg_{\bar{\theta}}(\bar{d}) - \bar{\theta}_{a,a'}\, \rk{\bar{d}}}\\
&= \prod_{(a,a') \in Q_0 \times Q_0'}\left[\det(g_{1,a})^{d_{a'}'\left(\deg_{\bar{\theta}}(\bar{d}) - \bar{\theta}_{a,a'}\, \rk{\bar{d}}\right)} \det(g_{2,a'})^{d_a\left(\deg_{\bar{\theta}}(\bar{d}) - \bar{\theta}_{a,a'}\, \rk{\bar{d}}\right)}\right]\\
&= \prod_{a\in Q_0} \det(g_{1,a})^{\rk{d'}\deg_{\bar{\theta}}(\bar{d)} - \rk{\bar{d}}\sum\limits_{a' \in Q_0'} d_{a'}' \bar{\theta}_{a,a'}} \,
\prod_{a'\in Q_0'} \det(g_{2,a'})^{\rk{d}\deg_{\bar{\theta}}(\bar{d}) - \rk{\bar{d}}\sum\limits_{a \in Q_0} d_a \bar{\theta}_{a,a'}}.
\end{align*}
\end{small}
Now, we have the exponent
\begin{align*}
&\rk{d'}\deg_{\bar{\theta}}(\bar{d}) - \rk{\bar{d}}\sum_{a' \in Q_0'} d_{a'}' \bar{\theta}_{a,a'}\\
= &\rk{d'}\left(\rk{d} \deg_{\theta'}(d') + \rk{d'} \deg_{\theta}(d)\right) - \rk{\bar{d}} \left(\theta_a\, \rk{d'} + \deg_{\theta'}(d')\right)\\
= &\rk{d'}\, \rk{d}\, \deg_{\theta'}(d') + \rk{d'}^2\, \deg_{\theta}(d) - \theta_a\, \rk{d}\, \rk{d'}^2 - \rk{d}\, \rk{d'}\, \deg_{\theta'}(d')\\
= &\rk{d'}^2\left[\deg_\theta(d) - \theta_a\, \rk{d}\right].
\end{align*}
Similarly,
$$\rk{d}\deg_{\bar{\theta}}(\bar{d}) - \rk{\bar{d}}\sum_{a \in Q_0} d_a \bar{\theta}_{a,a'} = \rk{d}^2[\deg_{\theta'}(d') - \theta_{a'}'\, \rk{d'}].$$
Therefore,
$$\chi_3(u(g_1,g_2)) = \chi_1(g_1)^{\rk{d'}^2}\, \chi_2(g_2)^{\rk{d}^2}.$$

Now let $x_i$ ($i=1,2$) be a $\chi_i$-semistable point in $\mathcal{A}_i$.  Then
there exists $f_i \in k[\mathcal{A}_i]^{G_i,\chi_i^{n_i}}$ such that
$f_i(x_i) \neq 0$ for some positive integer $n_i$.  Let $a = n_2\, \rk{d'}^2$
and $b = n_1\, \rk{d}^2$.  The map $\phi$ induces an isomorphism $\eta : \mathcal{A}_1 \times \mathcal{A}_2 \to \phi(\mathcal{A}_1 \times \mathcal{A}_2)$ and hence we get an isomorphism 
$\eta^* : k[\phi(\mathcal{A}_1 \times \mathcal{A}_2)] \to k[\mathcal{A}_1 \times \mathcal{A}_2] = k[\mathcal{A}_1] \otimes k[\mathcal{A}_2]$.  Let $f_3 \in k[\phi(\mathcal{A}_1 \times \mathcal{A}_2)]$ be such that
$\eta^*(f_3) = f_1^a \otimes f_2^b$.  Then we have

\begin{align*}
f_3\left((x_1\otimes x_2) u(g_1,g_2)\right) &= f_3 \circ\eta(x_1 g_1, x_2 g_2)
= \eta^*(f_3)(x_1 g_1, x_2 g_2)\\
&= f_1^a \otimes f_2^b(x_1 g_1, x_2 g_2)
= f_1^a(x_1 g_1)\, f_2^b(x_2 g_2)\\
&= \chi_1^{n_1n_2\rk{d'}^2}(g_1) f_1^a(x_1)\chi_2^{n_2n_1\rk{d}^2}(g_2) f_2^b(x_2)\\
&= \chi_3^{n_1n_2}(u(g_1,g_2))f_1^a\otimes f_2^b(x_1,x_2)\\
&= \chi_3^{n_1n_2}(u(g_1,g_2))\eta^*(f_3)(x_1,x_2) = \chi_3^{n_1n_2}(u(g_1,g_2))f_3(\eta(x_1,x_2))\\
&= \chi_3^{n_1n_2}(u(g_1,g_2))f_3(x_1\otimes x_2).
\end{align*}
It is clear that $f_3(x_1\otimes x_2) \neq 0$.

%

We have thus shown that 
$$f_3(\phi(x)\, u(g)) = \chi_3^{n_1n_2}(u(g))f_3(\phi(x)) \quad \mbox{for all} \quad x \in \mathcal{A}_1 \times \mathcal{A}_2, g \in G_1 \times G_2.$$
If we can extend $f_3$ to $\mathcal{A}_3$ such that
$$f_3(x\, g_3) = \chi_3^n(g_3)f_3(x) \quad \mbox{for all} \quad x \in \mathcal{A}_3, g_3 \in G_3$$
then we shall have the $\chi_3$-semistability of $x_1\otimes x_2$. 


\section{Relation between the bundles}
	\label{sec:3}
Let $u : G \to H$ be a morphism between algebraic groups over a field
$k$.  Suppose that we have a commutative diagram of morphisms of
$k$-varieties
\begin{equation*}
  \label{eq:E2}
  \xymatrix{%
    X \ar[r]^{\phi} \ar[d]_{p_X}
    & Y \ar[d]^{p_Y} \\
    M \ar[r]_{\bar{\phi}}
    & N,}
\end{equation*}
where $M$ is a good quotient of $X$ by $G$, and $N$ is a good
quotient of $Y$ by $H$, and $\phi$ is $u$-equivariant.  Let
$\pi_X : E_X \to X$ be a $G$-bundle and $\pi_Y : E_Y \to Y$ be an
$H$-bundle.  Suppose that $F_X$ is descent of $E_X$ on $M$ and $F_Y$
is descent of $E_Y$ on $N$, that is, $F_X$ and $F_Y$ are vector
bundles on $M$ and $N$ respectively such that $p_X^*F_X = E_X$ and
$p_Y^*F_Y = E_Y$.  The pullback bundle $\phi^*E_Y$ is naturally a
$G$-bundle as follows.  Let $(x,e_y) \in X \times E_Y$ be a
point on $\phi^*E_Y$.  Then, by definition $\phi(x) = \pi_Y(e_y)$.
For any $g \in G$, define $(x,e_y)\, g = (x\, g, e_y\, u(g))$.  Then,
we have $\phi(x\, g) = \phi(x)\, u(g)$, and as $E_Y$ is an
$H$-bundle, we get $\pi_Y(e_y\, u(g)) = \pi_Y(e_y)\, u(g) = 
\phi(x)\, u(g)$.  Thus, we get a structure of a $G$-bundle on
$\phi^*E_Y$.  Now, we have,
\begin{equation*}
\phi^*E_Y = \phi^*(p_Y^* F_Y) = (p_Y \circ \phi)^* F_Y =
 (\bar{\phi} \circ p_X)^* F_Y = p_X^*(\bar{\phi}^* F_Y).
\end{equation*}
Therefore, $\bar{\phi}^* F_Y$ is descent of the $G$-bundle
$\phi^*E_Y$.  Suppose that $\phi^*E_Y$ is isomorphic to $E_X$ as $G$
bundles on $X$.  Then, the vector bundles $(\bar{\phi}^* F_Y)$ and
$F_X$ on $M$, being the descent of isomorphic vector bundles, are
isomorphic.  The same is true in the holomorphic setting.

\subsection{Relation between the natural line bundles over \C}
	\label{sec:3.2}

In this subsection we assume that $k = \C$.

Let $N(Q,\theta,d)$ denote the complex space associated to the moduli
space $M(Q,\theta,d)$ of $\theta$-semistable representations of type
$d$ of a finite quiver $Q$.  Let $M_1 = M(Q,\theta,d)$,
$M_2 = M(Q',\theta',d')$ and
$M_3 = M(Q\otimes Q',\overline{\theta},\bar{d})$, and $N_i$
$(i = 1,2,3)$ be the complex space associated to $M_i$.

Let $\phi : \mathcal{A}_1 \times \mathcal{A}_2 \to \mathcal{A}_3$ be
the map defined by
\begin{equation} \label{eq:phi}
\phi((\rho_{\alpha})_{\alpha \in Q_1},
(\sigma_{\alpha'})_{\alpha' \in Q_1'})
= \Big((\rho_{\alpha} \otimes \id{W_{a'}})_{(\alpha,a') \in
Q_1\times Q_0'},(\id{V_a}\otimes \sigma_{\alpha'})_{(a,\alpha')
\in Q_0 \times Q_1'}\Big).
\end{equation}
The map $\phi$ is injective.  Also, let $u : G_1 \times G_2 \to G_3$
be the group homomorphism defined by
$u\Big(\big((g_a)_{a \in Q_0},(g_{a'}')_{a' \in Q_0'}\big)\Big) =
(g_a \otimes g_{a'}')_{(a,a') \in Q_0 \times Q_0'}$.  Then, the map
$\phi$ is $u$-equivariant, that is,
$\phi((\rho,\sigma)(g,g')) = \phi(\rho,\sigma)\, u(g,g')$ for all
$\rho\in \mathcal{A}, \sigma \in \mathcal{A}', g \in G_1$ and
$g' \in G_2$.

Let $\Delta_1 =
\set{ce \suchthat c \in \units{\C}, e=(\id{V_a})_{a\in Q_0}}$ and
$\bar{G}_1 = \Delta \backslash G_1$.  Then, the action of $G_1$ on
$\mathcal{A}_1$ induces an algebraic right action of $\bar{G}_1$ on
$\mathcal{A}_1$.  We similarly define $\Delta_2, \bar{G}_2$ and
$\Delta_3, \bar{G}_3$.  Then $u$ induces a group homomorphism
$u : \bar{G}_1 \times \bar{G}_2 \to \bar{G}_3$.


Recall that by Proposition \ref{prop:1.1}, we have a morphism of
varieties $\sstab{p_1} : \sstab{\mathcal{A}_1} \to M_1$.  This
morphism is a good quotient of $\sstab{\mathcal{A}_1}$ by $\bar{G}_1$,
and its restriction $\stab{p_1} : \stab{\mathcal{A}_1} \to \stab{M_1}$
is a geometric quotient of $\stab{\mathcal{A}_1}$ by $\bar{G}_1$.  We
similarly have the maps $\sstab{p_2} : \sstab{\mathcal{A}_2} \to M_2$
and $\sstab{p_3} : \sstab{\mathcal{A}_3} \to M_3$ and their
restrictions $\stab{p_2}$ and $\stab{p_3}$ to $\stab{\mathcal{A}_2}$
and $\stab{\mathcal{A}_3}$ respectively. 

\begin{remark}\label{rem:3.1}
Let $d$ be coprime and
$\theta$-coprime, and $d'$ be coprime and $\theta'$-coprime.  Let
$r_a$ ($a \in Q_0$) and $r_{a'}'$ ($a' \in Q_0'$) be integers such
that $\sum_{a \in Q_0} d_ar_a = 1$ and
$\sum_{a' \in Q_0'} d_{a'}'r_{a'}' = 1$.  Let
$\overline{r}_{(a,a')} = r_ar_{a'}'$, $\overline{\theta}_{(a,a')} =
\theta_a + \theta_{a'}'$ and
$\overline{d}_{(a,a')} = d_ad_{a'}'$ ($(a,a') \in Q_0\times Q_0'$).
Then,
\begin{equation*}
\sum_{(a,a') \in Q_0\times Q_0'}\overline{d}_{(a,a')} \overline{r}_{(a,a')} = \sum_{(a,a') \in Q_0\times Q_0'} d_ad_{a'}' r_ar_{a'}' = (\sum_{a \in Q_0} d_ar_a)(\sum_{a' \in Q_0'} d_{a'}'r_{a'}') = 1.
\end{equation*}
Hence $\overline{d}$ is also coprime.  By virtue of the property
(\ref{item:1.3.6}) in Section \ref{sec:1.3}, we can assume that
$\theta$ and $\theta'$ are chosen so that $\overline{d}$ is also
$\overline{\theta}$-coprime.  In this case, $M_i = \stab{M_i}$ 
($i = 1,2,3$) are fine moduli spaces.
Considering the underlying complex space structures, the morphisms
$\sstab{p_i} = \stab{p_i} = p_i : \stab{\mathcal{A}_i} \to N_i$
are principal $\overline{G_i}$-bundles. 
\end{remark}

\begin{remark}
Let $\phi$ also denote the restriction map
$\phi : \stab{\mathcal{A}_1} \times \stab{\mathcal{A}_2} \to
\stab{\mathcal{A}_3}$. (Note that this map is well defined by
Proposition \ref{prop:2.1})  Then, the map $p_3 \circ \phi :
\stab{\mathcal{A}_1} \times \stab{\mathcal{A}_2} \to N_3$
is constant on the orbits of the action of
$\bar{G}_1 \times \bar{G}_2$ on
$\stab{\mathcal{A}_1} \times \stab{\mathcal{A}_2}$, and since the
product $p_1 \times p_2 : \stab{\mathcal{A}_1} \times
\stab{\mathcal{A}_2} \to N_1 \times N_2$ of principal bundles is again
a principal bundle, there exists a holomorphic map
$\bar{\phi} : N_1 \times N_2 \to N_3$ which makes the diagram
\begin{equation*}
  \label{eq:E11}
  \xymatrix{%
    \stab{\mathcal{A}_1} \times \stab{\mathcal{A}_2}
     \ar[r]^{\phi} \ar[d]_{p_1 \times p_2}
    & \stab{\mathcal{A}_3} \ar[d]^{p_3} \\
    N_1 \times N_2 \ar[r]_{\bar{\phi}}
    & N_3}
\end{equation*}
commutative.
\end{remark}

Let $E_1 = \mathcal{A}_1 \times \C$ be the trivial line bundle on
$\mathcal{A}_1$, and $n$ an integer
such that $n(\mu_\theta(d)-\theta_a) \in \Z$ for all $a \in Q_0$.
Consider the action of $G_1$ on $E$ that is defined by
\begin{equation*}
  \label{eq:22}
  (\rho, v)\, g = (\rho\, g, \chi_1(g)^{-1}v)
\end{equation*}
for all $\rho \in \mathcal{A}_1$, $v \in \C$, and $g \in G_1$, where
$\chi_1 : G_1 \to \units{\C}$ is the character given by
\begin{equation*}
  \label{eq:21}
  \chi_1(g) =
  \prod_{a\in Q_0} \det(g_a)^{n(\mu_\theta(d)-\theta_a)}.
\end{equation*}
Then by \cite[Proposition 3.4 and Section 3c]{PD}, there exists a
holomorphic line bundle $F_1$ on $N_1$, such that $p_1^*F_1$ is
$\bar{G}_1$-isomorphic to $\restrict{E_1}{\stab{\mathcal{A}_1}}$.
Similarly, we have two line bundles $F_2$ on
$N_2$ and $F_3$ on $N_3$, which are descents of the trivial bundles
$E_2$ and $E_3$ respectively, equipped with the actions of the groups
$G_2$ and $G_3$ respectively, twisted by the characters
$\chi_2 : G_2 \to \units{\C}$ and $\chi_3 : G_3 \to \units{\C}$, that
are defined by 
$\chi_2(g') = \prod_{a' \in Q_0'} \det(g_{a'}')^{n'(\mu_{\theta'}(d')-\theta'_{a'})}$ and $\chi_3(k) = \prod_{(a,a') \in Q_0 \times Q_0'}
\det(k_{(a,a')})^{\bar{n}(\mu_{\overline{\theta}}(\bar{d})-\overline{\theta}_{(a,a')})}$ for some suitable $n'$ and $\bar{n}$,
namely an integer $n'$ such that
$n'(\mu_\theta'(d')-\theta_{a'}') \in \Z$ for all $a' \in Q_0'$, and
$\overline{n} = nn'$.  Note that, since
$\mu_{\overline{\theta}}(\bar{d}) =
\mu_\theta(d) + \mu_{\theta'}(d')$,
$\bar{n}(\mu_{\overline{\theta}}(\bar{d})-\overline{\theta}_{(a,a')}) \in \Z$ for all $(a,a') \in Q_0 \times Q_0'$.

\begin{proposition}
	\label{prop:3.1}
Let $F_i$ ($i = 1,2,3$) be the natural line bundles on $N_i$ as above.
Then, the line bundles $\bar{\phi}^*F_3$ and
$\pr{1}^*(F_1^{\otimes n'\, \rk{d'}}) \otimes \pr{2}^*(F_2^{\otimes n \, \rk{d}})$ on $N_1 \times N_2$
are isomorphic, where $\pr{j} : N_1 \times N_2 \to N_j$ ($j = 1,2$) is
the projection map. 
\end{proposition}

\proof  Recall that
$u : G_1 \times G_2 \to G_3$ is the group homomorphism defined by
$$u\Big(\big((g_a)_{a \in Q_0},(g_{a'}')_{a' \in Q_0'}\big)\Big) =
(g_a \otimes g_{a'}')_{(a,a') \in Q_0 \times Q_0'},$$ and the map
$\phi$ is $u$-equivariant.
Moreover, $u$ induces a group homomorphism
$u : \bar{G}_1 \times \bar{G}_2 \to \bar{G}_3$, and $\phi$ is
$u$-equivariant.  We note that, for any $g \in G_1$ and $g' \in G_2$,
we have
\begin{align*}
\chi_3(u(g,g')) &= 
\chi_3((g_a \otimes g_{a'}')_{(a,a') \in Q_0 \times Q_0'})\\
&= \prod_{(a,a') \in Q_0 \times Q_0'} (\det(g_a)^{d_{a'}'}
\det(g_{a'}')^{d_a})^{\bar{n}(\mu_{\theta}(d)-\theta_a +
\mu_{\theta'}(d')-\theta'_{a'})}\\
&= \prod_{(a,a') \in Q_0 \times Q_0'}
\Big((\det(g_a)^{\bar{n}d_{a'}'(\mu_{\theta}(d)-\theta_a)} 
\det(g_{a'}')^{\bar{n}d_a(\mu_{\theta'}(d')-\theta'_{a'})} \Big)
\times \\ 
&\qquad \prod_{(a,a') \in Q_0 \times Q_0'}
\Big((\det(g_a)^{\bar{n}d_{a'}'(\mu_{\theta'}(d')-\theta'_{a'})} 
\det(g_{a'}')^{\bar{n}d_a(\mu_{\theta}(d)-\theta_a)} \Big)\\
&= \prod_{a \in Q_0}\det(g_a)^{\bar{n}(\sum_{a' \in Q_0'}d_{a'}')(\mu_{\theta}(d)-\theta_a)}
\prod_{a' \in Q_0'} \det(g_{a'}')^{\bar{n}(\sum_{a \in Q_0}d_a)(\mu_{\theta'}(d')-\theta'_{a'})} \times\\
&\qquad \prod_{a \in Q_0} \det(g_a)^{\bar{n}\sum_{a' \in Q_0'}d_{a'}'(\mu_{\theta'}(d')-\theta'_{a'})}
\prod_{a' \in Q_0'} \det(g_{a'}')^{\bar{n}\sum_{a \in Q_0}d_a(\mu_{\theta}(d)-\theta_a)}\\
&= \chi_1(g)^{\frac{\bar{n}}{n} \rk{d'}} \chi_2(g')^{\frac{\bar{n}}{n'} \rk{d}}.
\end{align*}
Since we have $\bar{n} = nn'$, letting
$E_X = \pr{1}^*(E_1^{\otimes n'\, \rk{d'}})\otimes
\pr{2}^*(E_2^{\otimes n \, \rk{d}})$ with the action of
$G_1 \times G_2$ induced by the actions of $G_1$ and $G_2$ on $E_1$
and  $E_2$ respectively we see that the line bundle
$\phi^*(\restrict{E_3}{\stab{\mathcal{A}_3}})$ is isomorphic to
$\restrict{E_X}{\stab{\mathcal{A}_1} \times \stab{\mathcal{A}_2}}$ as
$\bar{G}_1 \times \bar{G}_2$-line bundles. (Here
$\pr{j} : \mathcal{A}_1 \times \mathcal{A}_2 \to \mathcal{A}_j$ is
the projection map for $j = 1,2$.)  It is easy to see that the line
bundle $\pr{1}^*(F_1^{\otimes n'\, \rk{d'}}) \otimes
\pr{2}^*(F_2^{\otimes n \, \rk{d}})$ is the descent of the line bundle
$E_X$.  Thus, by the observation in the beginning of the Section
\ref{sec:3}, we conclude the proof.

\subsection{Relation between the universal representations on fine moduli spaces over arbitrary fields}	\label{sec:3.1}
In this subsection, we assume that $k$ is any algebraically closed
field.

Recall from the item (\ref{item:1.3.4}) in Section \ref{sec:1.3} 
that the moduli space $\stab{M}(\theta,d)$ of $\theta$-stable
representations is a fine moduli space provided that the dimension
vector $d$ is coprime, that is, $\gcd(d_a \suchthat a \in Q_0) = 1$.
We describe here, how we get a universal representation in this case.
Since $d$ is coprime, there exist integers $r_a \in \Z$ ($a \in Q_0$),
such that $\sum_{a \in Q_0} d_a \, r_a = 1$.  For each $a \in Q_0$,
let $E_a = \stab{\mathcal{A}} \times V_a$ be the trivial vector bundle,
where $\stab{\mathcal{A}}$ (respectively, $\sstab{\mathcal{A}}$)
denotes the set of points $\rho \in \mathcal{A}$ such that the
representation $(V,\rho)$ is $\theta$-stable (respectively,
$\theta$-semistable).  Also for any $\alpha \in Q_1$, we have a
morphism of vector bundles
$\phi_\alpha : E_{\sr(\alpha)} \to E_{\tg(\alpha)}$ defined by
\begin{equation*}
\phi_\alpha(\rho,v) = (\rho, \rho_\alpha(v))
\end{equation*}
for all $(\rho,v) \in E_{\sr(\alpha)}$.
Let $\chi : G \to \units{k}$ be the character defined by
\begin{equation*}
\chi(g) = \prod_{a\in Q_0} \det(g_a)^{-r_a},
\end{equation*}
for all $g \in G$, and define an action of $G$ on $E_a$ by
\begin{equation*}
(\rho, v)\, g = (\rho \, g, \chi(g)\, g_a(v))
\end{equation*}
for all $\rho \in \stab{\mathcal{A}}, v \in V_a$ and $g \in G$.  Then
the stabiliser $\Delta$ of any $\rho \in \stab{\mathcal{A}}$ acts
trivially on $E_a$.  Hence $E_a$ descends to a vector bundle $U_a$ on
$\stab{M}(\theta,d)$.  We also note that for any
$\alpha \in Q_1, \rho \in \stab{\mathcal{A}}, v \in V_{\sr(\alpha)}$
and $g \in G$, we have
\begin{align*}
\phi_\alpha((\rho,v)\, g) &= \phi_\alpha(\rho\, g, \chi(g)g_{\sr(\alpha)}(v)) \\
&= (\rho\, g, (\rho\, g)_\alpha(\chi(g)g_{\sr(\alpha)}(v)))\\
&= (\rho\, g, g_{\tg(\alpha)}\circ \rho_\alpha \circ g_{\sr(\alpha)}^{-1}(\chi(g) g_{\sr(\alpha)}(v)))\\
&= (\rho\, g, \chi(g)g_{\tg(\alpha)}(\rho_\alpha(v))),
\end{align*}
and $(\phi_\alpha(\rho,v))\, g = (\rho, \rho_\alpha(v))\, g =
(\rho\, g, \chi(g)g_{\tg(\alpha)}(\rho_\alpha(v)))$.  Therefore
$\phi_\alpha$ is $G$-invariant.  Hence, $\phi_\alpha$ descends to a
morphism of vector bundles
$\psi_\alpha : U_{\sr(\alpha)} \to U_{\tg(\alpha)}$.  The family
$(U, \psi) = ((U_a)_{a\in Q_0}, (\psi_\alpha)_{\alpha \in Q_1})$ is a
universal family of $\theta$-stable representations of $Q$ of type
$d$.

Consider the assumptions as in Remark \ref{rem:3.1}.  Then we have the
following result.
\begin{proposition}
	\label{prop:3.2}
Let $(U_i,\psi_i)$ ($i = 1,2,3$) be the universal family of
representations on $N_i$.  Then, for each
$(a,a') \in Q_0 \times Q_0'$, the vector bundles
$\bar{\phi}^*U_{3,(a,a')}$ and
$\pr{1}^*(U_{1,a}) \otimes \pr{2}^*(U_{2,a'})$ are isomorphic, where
$\pr{j} : N_1 \times N_2 \to N_j$ ($j = 1,2$) is the projection map. 
\end{proposition}

\proof The proof is similar to the proof of Proposition
\ref{prop:3.1}.  The only difference is that the characters are given
by the integers $r_a, r_{a'}'$ and $r_a r_{a'}'$ as given in Remark
\ref{rem:3.1}, so that we have
\begin{align*}
\chi_3(u(g,g'))
&= \chi_3((g_a \otimes g_{a'}')_{(a,a') \in Q_0 \times Q_0'})\\
&= \prod_{(a,a') \in Q_0 \times Q_0'} {(\det(g_a)^{d_{a'}'}
\det(g_{a'}')^{d_a})}^{-r_ar_{a'}'}\\
&= \prod_{(a,a') \in Q_0 \times Q_0'}
\Big(\det(g_a)^{-r_a d_{a'}'r_{a'}'} 
\det(g_{a'}')^{-d_a r_ar_{a'}'} \Big)\\
&= \prod_{a \in Q_0}\det(g_a)^{-r_a(\sum_{a' \in Q_0'}d_{a'}'r_{a'}')}
\prod_{a' \in Q_0'} \det(g_{a'}')^{-(\sum_{a \in Q_0}d_ar_a)r_{a'}'}\\
&= \prod_{a \in Q_0}\det(g_a)^{-r_a} \prod_{a' \in Q_0'} \det(g_{a'}')^{-r_{a'}'}\\
&= \chi_1(g)\chi_2(g').  \hspace*{10cm} \qed
\end{align*}

\section{Covering quiver} \label{sec:4}
Given a quiver $Q = (Q_0,Q_1,\sr_Q,\tg_Q)$ and a dimension vector
$d = (d_a)_{a\in Q_0}$ of $Q$, we define the \emph{covering quiver}
(see \cite[Section 4.2]{Fra}) of $Q$ to be a quiver
$\tilde{Q}=(\tilde{Q}_0,\tilde{Q}_1,\sr_{\tilde{Q}},\tg_{\tilde{Q}})$, 
where
\begin{align*}
\tilde{Q}_0 &= \set{(a,\mu) \in Q_0\times \N \suchthat 1\leq \mu \leq d_a} \\
\tilde{Q}_1 &= \set{(\alpha,\mu,\nu) \in Q_1\times \N \times \N \suchthat 1\leq \mu \leq d_{\sr_Q(\alpha)}, 1\leq \nu \leq d_{\tg_Q(\alpha)}},
\end{align*}
and the maps $\sr_{\tilde{Q}}$ and $\tg_{\tilde{Q}}$ are defined by
\begin{equation*}
\sr_{\tilde{Q}}(\alpha,\mu,\nu) = (\sr_Q(\alpha),\mu), \quad \tg_{\tilde{Q}}(\alpha,\mu,\nu) = (\tg_Q(\alpha),\nu).
\end{equation*}
There is a projection map  of quivers
$\psi^Q = (\psi_0^Q, \psi_1^Q) : \tilde{Q} \to Q$ defined as
$\psi_0^Q : \tilde{Q}_0 \to Q_0 ; (a, \mu) \mapsto a$ and
$\psi_1^Q : \tilde{Q}_1 \to Q_1 ; (\alpha,\mu,\nu) \mapsto \alpha$.

Let $Q = (Q_0,Q_1,\sr_Q,\tg_Q)$ and
$Q' = (Q_0',Q_1',\sr_{Q'},\tg_{Q'})$ be two quivers.
A \emph{morphism} from $Q$ to $Q'$ is a pair $\phi=(\phi_0,\phi_1)$,
where $\phi_0: Q_0\to Q_0'$ and $\phi_1: Q_1\to Q_1'$ are functions,
such that
\begin{equation*}
  \label{eq:139}
  \phi_0(\sr_Q(\alpha)) = \sr_{Q'}(\phi_1(\alpha)), \quad
  \phi_0(\tg_Q(\alpha)) = \tg_{Q'}(\phi_1(\alpha))
\end{equation*}
for every arrow $\alpha$ of $Q$.  If $Q''$ is another quiver, and if
$\phi'$ a morphism from $Q'$ to $Q''$, then the pair
$\phi'\circ \phi = (\phi'_0\circ \phi_0, \phi'_1\circ \phi_1)$ is a
morphism from $Q$ to $Q'$, and is called the \emph{composite} of
$\phi$ and $\phi'$.  We thus get a category $\mathbf{Quiv}$, whose
objects are quivers, and whose morphisms are the morphisms defined
above.  An \emph{isomorphism} in this category is a morphism
$\phi=(\phi_0,\phi_1)$ such that $\phi_0,\phi_1$ are bijections.

Let us fix a dimension vector $d = (d_a)_{a\in Q_0}$ for the quiver 
$Q$ and a dimension vector $d' = (d_{a'}')_{a'\in Q_0'}$ for the
quiver $Q'$.  Then we have the following.

\begin{proposition}
	\label{prop:4.1}
The quiver $\tilde{Q}\otimes \tilde{Q'}$ is isomorphic to a subquiver
of $\widetilde{Q\otimes Q'}$ in the category $\mathbf{Quiv}$.
\end{proposition}

\begin{proof} We have
\begin{equation*}
(\widetilde{Q\otimes Q'})_0 = \set{\left((a,a'),\bar{\mu}\right) \in (Q_0\times Q_0') \times \N \suchthat 1\leq \bar{\mu} \leq d_a d_{a'}'}, \quad (\widetilde{Q\otimes Q'})_1 = A \sqcup B,
\end{equation*}
where
\begin{align*}
A &= \set{\left((a,\alpha'),\bar{\mu},\bar{\nu}\right) \in (Q_0 \times Q_1') \times \N \times \N \suchthat 1\leq \bar{\mu} \leq d_a d_{\sr_{Q'}(\alpha')}', 1\leq \bar{\nu} \leq d_a d_{\tg_{Q'}(\alpha')}'},\\
B &= \set{\left((\alpha,a'),\bar{\mu},\bar{\nu}\right) \in (Q_1 \times Q_0') \times \N \times \N \suchthat 1\leq \bar{\mu} \leq d_{\sr_Q(\alpha)} d_{a'}', 1\leq \bar{\nu} \leq d_{\tg_Q(\alpha)} d_{a'}'}
\end{align*}
and the maps $\sr_{\widetilde{Q\otimes Q'}}$ and
$\tg_{\widetilde{Q\otimes Q'}}$ are defined by
\begin{align*}
\sr_{\widetilde{Q\otimes Q'}}\left(\left((a,\alpha'),\bar{\mu},\bar{\nu}\right)\right) &= 
\left(\sr_{Q\otimes Q'}((a,\alpha')),\bar{\mu}\right) = \left((a,\sr_{Q'}(\alpha')),\bar{\mu}\right),\\
\sr_{\widetilde{Q\otimes Q'}}\left(\left((\alpha,a'),\bar{\mu},\bar{\nu}\right)\right) &= 
\left(\sr_{Q\otimes Q'}((\alpha,a')),\bar{\mu}\right) = \left((\sr_{Q}(\alpha),a'),\bar{\mu}\right),\\
\tg_{\widetilde{Q\otimes Q'}}\left(\left((a,\alpha'),\bar{\mu},\bar{\nu}\right)\right) &= \left((a,\tg_{Q'}(\alpha')),\bar{\nu}\right),\\
\tg_{\widetilde{Q\otimes Q'}}\left(\left((\alpha,a'),\bar{\mu},\bar{\nu}\right)\right) &= \left((\tg_{Q}(\alpha),a'),\bar{\nu}\right)
\end{align*}
for $\left((a,\alpha'),\bar{\mu},\bar{\nu}\right) \in A$ and
$\left((\alpha,a'),\bar{\mu},\bar{\nu}\right) \in B$.\\

Now, we consider the quiver $\tilde{Q}\otimes \tilde{Q'}$.  We have
\begin{align*}
(\tilde{Q}\otimes \tilde{Q'})_0 &= (\tilde{Q})_0 \times (\tilde{Q'})_0 = \set{\left((a,\mu),(a',\mu')\right) \suchthat a \in Q_0, 1 \leq \mu \leq d_a, a' \in Q_0', 1 \leq \mu' \leq d_{a'}'},\\
(\tilde{Q}\otimes \tilde{Q'})_1 &= (\tilde{Q})_0 \times (\tilde{Q'})_1  \sqcup (\tilde{Q})_1 \times (\tilde{Q'})_0= C \sqcup D,
\end{align*}
where
\begin{align*}
C &= \set{((a,\mu),(\alpha',\mu',\nu')) \suchthat a \in Q_0, 1 \leq \mu \leq d_a, \alpha' \in Q_1', 1 \leq \mu' \leq \sr_{Q'}(\alpha'), 1 \leq \nu' \leq \tg_{Q'}(\alpha')},\\
D &= \set{\left((\alpha,\mu,\nu),(a',\mu')\right) \suchthat \alpha \in Q_1, 1 \leq \mu \leq \sr_Q(\alpha), 1 \leq \nu \leq \tg_{Q}(\alpha), a' \in Q_0', 1 \leq \mu' \leq d_{a'}'},
\end{align*}
and the maps $\sr_{\tilde{Q}\otimes \tilde{Q'}}$ and
$\tg_{\tilde{Q}\otimes \tilde{Q'}}$ are defined by
\begin{align*}
\sr_{\tilde{Q}\otimes \tilde{Q'}}\left(\left((a,\mu),(\alpha',\mu',\nu')\right)\right)&=\left((a,\mu),\sr_{\tilde{Q'}}((\alpha',\mu',\nu'))\right)&&=\left((a,\mu),(\sr_{Q'}(\alpha'),\mu')\right),\\
\sr_{\tilde{Q}\otimes \tilde{Q'}}\left(\left((\alpha,\mu,\nu),(a',\mu')\right)\right)&=\left(\sr_{\tilde{Q}}((\alpha,\mu,\nu)),(a',\mu')\right)&&=\left((\sr_Q(\alpha),\mu),(a',\mu')\right),\\
\tg_{\tilde{Q}\otimes \tilde{Q'}}\left(\left((a,\mu),(\alpha',\mu',\nu')\right)\right) &=\left((a,\mu),(\tg_{Q'}(\alpha'),\nu')\right),&& \\
\tg_{\tilde{Q}\otimes \tilde{Q'}}\left(\left((\alpha,\mu,\nu),(a',\mu')\right)\right) &=\left((\tg_Q(\alpha),\nu),(a',\mu')\right). &&
\end{align*}

Next, for each $(a,a') \in Q_0\times Q_0'$, we fix a bijection
$f_{a,a'} : \set{1,2,\dotsc, d_a}\times \set{1,2,\dotsc, d_{a'}'} \to \set{1,2,\dotsc, d_a d_{a'}'}$.  Then we define 
$$\phi_0 : (\tilde{Q}\otimes \tilde{Q'})_0 \to (\widetilde{Q\otimes Q'})_0$$
by
\begin{equation*}
\phi_0\left(\left((a,\mu),(a',\mu')\right)\right) = \left((a,a'),f_{a,a'}(\mu,\mu')\right)
\end{equation*}
for $\left((a,\mu),(a',\mu')\right) \in (\tilde{Q}\otimes \tilde{Q'})_0$, and
$$\phi_1 : (\tilde{Q}\otimes \tilde{Q'})_1 \to (\widetilde{Q\otimes Q'})_1$$
by
\begin{equation*}
\phi_1\left(\left((a,\mu),(\alpha',\mu',\nu')\right)\right)=\left((a,\alpha'),f_{a,\sr_{Q'}(\alpha')}(\mu,\mu'),f_{a,\tg_{Q'}(\alpha')}(\mu,\nu')\right)
\end{equation*}
for $\left((a,\mu),(\alpha',\mu',\nu')\right) \in C$ and
\begin{equation*}
\phi_1\left(\left((\alpha,\mu,\nu),(a',\mu')\right)\right)=\left((\alpha,a'),f_{\sr_{Q}(\alpha),a'}(\mu,\mu'),f_{\tg_{Q}(\alpha),a'}(\nu,\mu')\right)
\end{equation*}
for $\left((\alpha,\mu,\nu),(a',\mu')\right) \in D$.  Then we have
\begin{align*}
\sr_{\widetilde{Q\otimes Q'}}\left(\phi_1\left(((a,\mu),(\alpha',\mu',\nu'))\right)\right) &= \sr_{\widetilde{Q\otimes Q'}}\left(\left((a,\alpha'),f_{a,\sr_{Q'}(\alpha')}(\mu,\mu'),f_{a,\tg_{Q'}(\alpha')}(\mu,\nu')\right)\right)\\
&= \left((a,\sr_{Q'}(\alpha')),f_{a,\sr_{Q'}(\alpha')}(\mu,\mu')\right)
\end{align*}
and
\begin{align*}
\phi_0\left(\sr_{\tilde{Q}\otimes \tilde{Q'}}\left(((a,\mu),(\alpha',\mu',\nu'))\right)\right) &= \phi_0\left(((a,\mu),(\sr_{Q'}(\alpha'),\mu'))\right)\\
&= \left((a,\sr_{Q'}(\alpha')),f_{a,\sr_{Q'}(\alpha')}(\mu,\mu')\right).
\end{align*}
Similarly for the other arrows and for the target map.
Therefore $\phi = (\phi_0,\phi_1)$ is a morphism of quivers from
$\tilde{Q}\otimes \tilde{Q'}$ to $\widetilde{Q\otimes Q'}$.  
Moreover, the map $\phi_0$ is a bijection and compatible with the
projection maps of the respective covering quivers 
\begin{equation}\label{covering_projection}
\psi_0^{Q \otimes Q'} \circ \phi_0 = \psi_0^Q \otimes \psi_0^{Q'}.
\end{equation}

We also have $\phi_1(C) \subset A$ and $\phi_1(D) \subset B$, and
$\restrict{\phi_1}{C}$ and $\restrict{\phi_1}{D}$ are injective.  It
follows that $\phi_1$ is injective. This completes the proof.
\end{proof}

\begin{remark}
We have
\begin{equation*}
\card{A} = \sum_{a \in Q_0} \sum_{\alpha' \in Q_1'} d_a^2 d_{\sr_{Q'}(\alpha')}' d_{\tg_{Q'}(\alpha')}' \neq (\sum_{a\in Q_0} d_a)(\sum_{\alpha' \in Q_1'} d_{\sr_{Q'}(\alpha')}' d_{\tg_{Q'}(\alpha')}') = \card{C}.
\end{equation*}
Thus, $\phi_1$ is in general not surjective.  Thus the quivers
$\tilde{Q} \otimes \tilde{Q'}$ and $\widetilde{Q\otimes Q'}$ are in
general not isomorphic.  They are isomorphic if only if the quivers
$Q$ and $Q'$ are already covering quivers.
\end{remark}

Let $k$ be an arbitrary field.  We fix a $k$-vector space $V_a$ and an ordered basis for $V_a$ for each $a \in Q_0$.
Given a representation $V = (V,\rho)$ of the quiver $Q$, we obtain a
representation $\tilde{V} = (\tilde{V}, \tilde{\rho})$ of the covering
quiver $\tilde{Q}$ as follows.

For each $(a,\mu) \in \tilde{Q}_0$ we set $\tilde{V}_{(a,\mu)} = $ the
subspace of $V_a$ generated by the $\mu$-th basis element of $V_a$,
and for each $(\alpha,\mu,\nu) \in \tilde{Q}_1$ we set
$\tilde{\rho}_{(\alpha,\mu,\nu)} : \tilde{V}_{(\sr_Q(\alpha),\mu)}
\to \tilde{V}_{(\tg_Q(\alpha),\nu)}$ to be the
$(\mu,\nu)-th$ entry of the matrix corresponding to $\rho_\alpha$.
We shall call $\tilde{V}$ the covering representation associated to the
representation $V$. Similarly, we have an obvious way of associating a
representation $V$ of $Q$ to a representation $\tilde{V}$ of
$\tilde{Q}$.  We thus get an identification
$\varphi : \mathcal{A}(\tilde{Q},\mathbf{1}) \rightarrow
\mathcal{A}(Q,d)$ of the representation spaces.

Associated to each weight $\theta$ of a quiver $Q$, there is a weight
$\tilde{\theta}$ for the covering quiver $\tilde{Q}$ satisfying
$\tilde{\theta}_{(a, \mu)} = \theta_{\psi_0(a, \mu)} = \theta_a$ for
all $a \in Q_0$. Using this identity of weights we can check that
$\mu_{\tilde{\theta}} (\tilde{V}) = \mu_\theta(V)$.
Now by comparing the forbidden sub-dimension vectors of $V$ and
$\tilde{V}$ we can prove the following result. 
\begin{lemma}
If $V$ is a representation of $Q$, then $V$ is $\mu_\theta$-semistable
if and only if $\tilde{V}$ is $\mu_{\tilde{\theta}}$-semistable.
\end{lemma}
\begin{proof}
Since there is a bijection between the forbidden dimension vectors
\cite[Pg. 1212]{Fra} and $\mu_{\tilde{\theta}} (\tilde{V}) =
\mu_\theta(V)$, it is enough to prove that there is a bijection
between sub-representation realising forbidden dimension vectors. 
 
If $W$ is a sub-representation of $V$ realising forbidden sub-dimension
vector then $\tilde{W}$ is a sub-representation of $\tilde{V}$ realising
the corresponding forbidden sub-dimension vector.  Conversely, we can
define sub-representation using
$W_a := \oplus_{(a, \mu) \in \psi_0^{-1}(a)} \tilde{W}_{(a, \mu)}$
for any $\tilde{W} \subseteq \tilde{V}$.
\end{proof}


Let $\stab{\tilde{\mathcal{A}}}$ be the stable locus in the
representation space
$\tilde{\mathcal{A}} = \mathcal{A}(\tilde{Q}, \bf{1})$ and
$T = \Delta'\backslash G(\tilde{Q},\bf{1}) = \Delta'\backslash (\prod_{(a,\mu)} \mathbb{G}_m)$
be the corresponding group, where $\Delta'$ is the one dimensional
subgroup $\{(\lambda)_{(a,\mu)} \suchthat \lambda \in k^\times\}$.
Then comparing the GIT quotients we get the following relation between
the moduli spaces.

\begin{proposition}
	\label{prop:4.4}
There is a surjective map
$\tau_N : \tilde{N} \to N\, ; \, [\tilde{V}] \mapsto [V]$
satisfying
$$p \circ \varphi = \tau_N \circ \tilde{p},$$
where $\tilde{p} : \stab{\tilde{\mathcal{A}}} \to \tilde{N}$ is the GIT
quotient by the group $T$ and $\varphi$ is the identification of
$\stab{\tilde{\mathcal{A}}}$ and $\stab{\mathcal{A}}$. 
\end{proposition}

%

Let $Q = (Q_0,Q_1,s,t)$ be a quiver, and $Q'$ be a subquiver of $Q$
with $Q_0' = Q_0$, that is, $Q$ and $Q'$ have the same set of vertices.
Let $d \in \N^{Q_0}$ be a fixed dimension vector.  For each $a \in Q_0$
we also fix a vector space of dimension $d_a$.  Let
\begin{equation*}
\mathcal{A} = \mathcal{A}(Q,d) = \bigoplus_{\alpha \in Q_1} \Hom[k]{V_{\sr(\alpha)}}{V_{\tg(\alpha)}}
\end{equation*}
and
\begin{equation*}
\mathcal{A}' = \mathcal{A}(Q',d) = \bigoplus_{\alpha \in Q_1'} \Hom[k]{W_{\sr(\alpha)}}{W_{\tg(\alpha)}}
\end{equation*}
be the representation spaces of $Q$ and $Q'$ respectively of type $d$.
We have natural linear maps
\begin{equation*}
i_d : \mathcal{A}' \to \mathcal{A} \quad \mbox{and} \quad \pi_d : \mathcal{A} \to \mathcal{A}'
\end{equation*}
defined by
\begin{equation*}
i_d(\sigma)_\alpha =
\begin{cases}
\sigma_\alpha \quad \mbox{if} \quad \alpha \in Q_1' \\
0 \quad \mbox{otherwise}
\end{cases}
\end{equation*}
and
\begin{equation*}
\pi_d(\rho)_\alpha = \rho_\alpha \quad \mbox{for all} \quad \alpha \in Q_1'. 
\end{equation*}
Then we have
\begin{equation}
	\label{eq:101}
\pi_d \circ i_d = \id{\mathcal{A}'}.
\end{equation}

Given any $\rho \in R$, the pair $(V,\rho)$ is a representation of $Q$
of type $d$ and $(V,\pi_d(\rho))$ is a representation of $Q'$ of type
$d$. Similarly, given any $\tau \in R'$, the pair $(V,\tau)$ is a
representation of $Q'$ of type $d$ and $(V,i_d(\tau))$ is a
representation of $Q$ of type $d$.

\begin{lemma}
	\label{lem:3}
\begin{enumerate}
\item A representation $(V,\rho)$ of $Q$ of type $d$ is
$\theta$-semistable if the representation $(V,\pi_d(\rho))$ of $Q'$ is
$\theta$-semistable. \label{item:4.3.1}
\item A representation $(V,\tau)$ of $Q$ of type $d$ is
$\theta$-semistable if and only if the representation $(V,i_d(\tau))$
of $Q'$ is $\theta$-semistable. \label{item:4.3.2}
\end{enumerate}
\end{lemma}

\begin{proof}
We note that if $(W, \sigma)$ is a subrepresentation of $(V,\rho)$ of
type $d'$, then the representation $(W, \pi_{d'}(\sigma))$ of $Q'$ is
a subrepresentation of $(V,\pi_d(\rho))$ of type $d'$.  Thus
(\ref{item:4.3.1}) follows.

Similarly, if $(W,\kappa)$ is a subrepresentation of $(V,\tau)$ of
type $d'$ then $(W,i_{d'}(\kappa))$ is a subrepresentation of
$(V,i_d(\tau))$ of type $d'$. Therefore it follows that 
\begin{equation}
(V,i_d(\tau)) \quad {\rm is~} \theta-{\rm semistable} \Rightarrow
 (V,\tau) \quad {\rm is~} \theta-{\rm semistable}.
\end{equation}
We also have
$$(V, \tau) = (V, \pi_d \circ i_d(\tau)).$$
Therefore by (\ref{item:4.3.1}),
\begin{equation}
(V,\tau) \quad {\rm is~} \theta-{\rm semistable} \Rightarrow
 (V,i_d(\tau)) \quad {\rm is~} \theta-{\rm semistable}.
\end{equation}
This proves (\ref{item:4.3.2}).
\end{proof}

We will now consider the special case where
$\tilde{Q} \otimes \tilde{Q'}$ can be identified with a sub-quiver of
$\widetilde{Q \otimes Q'}$ via the map $\phi$.  We can define the
forgetful functor $\mathbf{U} :  \Repr[k]{\widetilde{Q \otimes Q'}}
\to \Repr[k]{\tilde{Q} \otimes \tilde{Q'}}$ and the extension by zero
functor $\mathbf{I} : \Repr[k]{\tilde{Q} \otimes \tilde{Q'}}
\to \Repr[k]{\widetilde{Q \otimes Q'}}$.  Now using Lemma \ref{lem:3}
and the commutativity of the diagram
\begin{equation*}
\begin{tikzcd}[row sep = large]
  \mathcal{A}(Q,d)\times \mathcal{A}(Q',d') \arrow[r,equal] \arrow[d] 
  & \mathcal{A}(\tilde{Q},\id{})\times \mathcal{A}(\tilde{Q'},\id{}) \arrow[r]
  & \mathcal{A}(\tilde{Q}\otimes \tilde{Q'},\id{}) \arrow[d]
\\
  \mathcal{A}(Q\otimes Q',\bar{d}) \arrow[rr,equal] & &
  \mathcal{A}(\widetilde{Q\otimes Q'},\id{}),
\end{tikzcd}
\end{equation*}
we get the following result.
\begin{proposition}
If $V$ and $W$ are representations of $Q$ and $Q'$ respectively, then 
\begin{enumerate}
\item If $\mathbf{U}(\widetilde{V \otimes W})$ is semistable then so is
$\widetilde{V\otimes W}$.
\item $\mathbf{I}(\tilde{V} \otimes \tilde{W})$ is semistable if and
only if $\widetilde{V \otimes W}$ is semistable.
\end{enumerate}
\end{proposition}

\begin{corollary} \label{cor:4.7}
 There is a commutative diagram

\[
    \begin{tikzcd}[row sep=1.5em, column sep = 1.5em]
    \stab{\tilde{\mathcal{A}_1}} \times \stab{\tilde{\mathcal{A}_2}} \arrow[rr] \arrow[dr,equal,swap] \arrow[dd,swap] &&
    \stab{\tilde{\mathcal{A}_3}} \arrow[dd] \arrow[dr,equal] \\
    & \stab{\mathcal{A}_1} \times \stab{\mathcal{A}_2} \arrow[dd,"p_1 \times p_2" near start] \arrow[rr] &&
    \stab{\mathcal{A}_3} \arrow[dd,"p_3"] \\
    \tilde{N_1} \times \tilde{N_2} \arrow[rr,"\tilde{\phi}" near start] \arrow[dr,swap,"\tau_1 \times \tau_2"] && \tilde{N_3} \arrow[dr,swap,"\tau_3"] \\
    & N_1 \times N_2 \arrow[rr,swap,"\bar{\phi}"] && N_3
    \end{tikzcd}.
\]
\end{corollary}

\subsection{Relation between natural line bundles on the moduli of covering quivers over \C}
Let $k = \C$.  Let $\tilde{F_i}$ $(i = 1,2,3)$ denote the natural line
bundle on the moduli space $\tilde{N_i}$, and
$\tau_i := \tau_{N_i} : \tilde{N_i} \to N_i$ be the morphism as in
Proposition \ref{prop:4.4}. 
Let $u_i : T_i \to \bar{G_i}$ be the corresponding group homomorphism.
Then it can be seen that 
\begin{equation*}
\varphi_i(\tilde{\rho_i}\cdot \tilde{g_i}) =
 \varphi_i(\tilde{\rho_i}) \cdot u_i(\tilde{g_i})
\end{equation*}
for all $\tilde{\rho_i} \in \stab{\tilde{\mathcal{A}_i}}$ and
$\tilde{g_i} \in T_i$.  Moreover the $T_i$-bundle
$\stab{\tilde{\mathcal{A}_i}} \times \C$ and the pullback bundle
$\varphi_i^*(\stab{\mathcal{A}_i} \times \C)$ are isomorphic as
$T_i$-bundles because the characters do not change.  Hence by the
discussion in the beginning of the Section \ref{sec:3}, we get the
following.

\begin{proposition}
	\label{prop:4.5}
The line bundles $\tilde{F}_i$ and $\tau_i^*F_i$ are isomorphic.
\end{proposition}

\begin{corollary}
The line bundles $\tilde{\phi}^*\tilde{F_3}$ and
$\widetilde{\pr{1}}^*(\tilde{F_1}^{\otimes n'\, \rk{d'}}) \otimes \widetilde{\pr{2}}^*(\tilde{F_2}^{\otimes n \, \rk{d}})$ on $\tilde{N_1} \times \tilde{N_2}$ are isomorphic, where
$\widetilde{\pr{j}} : \tilde{N_1} \times \tilde{N_2} \to \tilde{N_j}$
($j = 1,2$) is the projection map.
\end{corollary}
 
\proof We have
\begin{equation*}
\pr{i} \circ (\tau_1 \times \tau_2) = \tau_i \circ \widetilde{\pr{i}}
\end{equation*}
for $i = 1,2$.  Therefore
\begin{equation*}
(\tau_1 \times \tau_2)^*(\pr{i}^*F_i) = (\pr{i} \circ (\tau_1 \times \tau_2))^*F_i = (\tau_i \circ \widetilde{\pr{i}})^*F_i = \widetilde{\pr{i}}^*(\tau_i^*F_i) = \widetilde{\pr{i}}^*\tilde{F_i}.
\end{equation*}
Taking tensor products, and using Proposition \ref{prop:3.1} and Corollary \ref{cor:4.7} we get
\begin{align*}
\widetilde{\pr{1}}^*(\tilde{F_1}^{\otimes n'\, \rk{d'}}) \otimes \widetilde{\pr{2}}^*(\tilde{F_2}^{\otimes n \, \rk{d}}) &= (\tau_1 \times \tau_2)^*(\pr{1}^*F_1^{\otimes n'\, \rk{d'}} \otimes \pr{2}^*F_2^{\otimes n \, \rk{d}})\\
&= (\tau_1 \times \tau_2)^*\bar{\phi}^*F_3 = (\tau_3 \circ \tilde{\phi})^*F_3 = \tilde{\phi}^* \tau_3^* F_3 = \tilde{\phi}^* \tilde{F_3.} \qed
\end{align*}
 
\subsection*{Acknowledgment} A part of this manuscript constitutes a
chapter of the first author's Ph.D. thesis written at the Harish-Chandra
Research Institute, Prayagraj where he was a student. Some part of this
manuscript was written when the first author was a postdoctoral fellow at
the Tata Institute of Fundamental Research, Mumbai. He would like to
extend his gratitude to these two institutes.  We thank Yiqiang Li for
informing us about his work related to product valued quiver. We also
thank the anonymous referees for their constructive comments which led
to the improvements in the manuscript.

%


\begin{thebibliography}{9}
\bibitem{CP} Luis \'Alvarez C\'onsul and Oscar Garc\'ia-Prada, \emph{Hitchin-Kobayashi correspondence, quivers, and vortices}, Comm.\ Math.\ Phys., 238(1--2):1--33, 2003.

\bibitem{PD} Pradeep Das,
  \emph{A natural Hermitian line bundle on the moduli space of
  semistable representations of a quiver}, Indian J.\ Pure Appl.\ Math., \textbf{51}(3): 1003--1021, 2020.
  
\bibitem{DMR} Pradeep Das, S.~Manikandan and N.~Raghavendra,
  \emph{Holomorphic aspects of moduli of representations of quivers},
  Indian J.\ Pure Appl.\ Math., \textbf{50}(2): 549--595, 2019.
  
\bibitem{FW}
Gerd Faltings and Gisbert W\"{u}stholz,
\emph{Diophantine approximations on projective spaces},
Invent.\ Math., \textbf{116}(1-3):109--138, 1994.

\bibitem{Fra}
H.~Franzen, \emph{Chow rings of fine quiver moduli are 
tautologically presented}, Math.\ Z., \textbf{279} (2015), no. 3-4,
1197--1223. 

\bibitem{Hos} V.~Hoskins, \emph{Parallels between moduli of quiver representations and vector bundles over curves}, SIGMA 14 (2018), Paper No.~127, 46 pp. 
  
\bibitem{Kel} Bernhard Keller,
  \emph{The periodicity conjecture for pairs of Dynkin diagrams},
  Annals of Mathematics, \textbf{177} (2013), 111--170.
  
\bibitem{Kin} A.~D.~King, \emph{Moduli of representations of
  finite dimensional algebras}, Quart.\ J.\ Math.\ Oxford \textbf{45}
  (1994), 515--530.
  
\bibitem{KW}
A.~D.~King and Charles H.~Walter, \emph{On Chow rings of 
fine moduli spaces of modules}, J.\ Reine Angew.\ Math.\  \textbf{461} 
(1995), 179--187.

\bibitem{LL}
Yiqiang Li and Zongzhu Lin, \emph{A realization of quantum groups via product valued quivers}, Algebr.\ Represent.\ Theory, \textbf{13} (2010), 427--444.

\bibitem{NS}
M.~S. Narasimhan and C.~S. Seshadri,
\emph{Stable and unitary vector bundles on a compact Riemann surface}, Ann.\ of Math.\ (2), \textbf{82}:540--567, 1965.
  
\bibitem{Re08} Markus Reineke, \emph{Moduli of Representations of Quivers}, in : Trends in Representation Theory of Algebras and Related Topics, 589-638, EMS Series of Congress Reports, Z\"urich, 2008.
		 
\bibitem{Tot}
Burt Totaro, \emph{Tensor products of semistables are semistable},
In {\em Geometry and analysis on complex manifolds}, pages 242--250. World Sci. Publ., River Edge, NJ, 1994.

\bibitem{Wei}
Thorsten ~Weist, \emph{Localization in quiver moduli spaces}, Represent. Theory \textbf{17}(13), 382--425, 2013.

\end{thebibliography}
\end{document}